\documentclass{article}

\usepackage{amsmath}
\usepackage{amssymb}
\usepackage{amsthm}
\usepackage{bm}
\usepackage{enumerate}
\usepackage{enumitem}
\usepackage{esint}
\usepackage[T1]{fontenc}
\usepackage[margin=0.75in]{geometry}
\usepackage{graphicx}
\usepackage[utf8]{inputenc}
\usepackage{mathtools}
\usepackage{subcaption}
\newtheorem{theorem}{Theorem}[section]

\newtheorem{proposition}[theorem]{Proposition}

\newcommand{\grad}[1]{\nabla#1}

\newcommand{\norm}[2][]{\bigl\|#2\bigr\|_{#1}}

\newcommand{\vs}[1]{\mathcal{#1}}
\usepackage{enumitem}
\usepackage{verbatim}
\usepackage{scalerel}

\usepackage{xcolor}
\usepackage{hyperref}

\numberwithin{equation}{section}

\title{Global solutions to the Boussinesq equations and the planetary geostrophic approximation}
\author{Ryan Budahazy\footnote{\href{mailto:rbudahazy@tamu.edu}{rbudahazy@tamu.edu}, Department of Mathematics, Texas A\&M University, College Station, TX--77843}, \qquad Xin Liu\footnote{\href{mailto:xliu23@tamu.edu}{xliu23@tamu.edu}, Department of Mathematics, Texas A\&M University, College Station, TX--77843}}
\date{\today}

\begin{document}
%\doublespacing

\maketitle

\begin{abstract}
    We investigate the three dimensional Boussinesq equations in the planetary geostrophic scale, where the buoyancy/gravity driven flow is near geostrophic balance in the large horizontal spatial scale. By establishing the uniform (with respect to the characteristic number) estimates, we are able to obtain the global weak solutions and the unique global strong solutions to the Boussinesq equations without restriction on the initial data. Furthermore, sending the characteristic number to zero with the established compactness rigorously justifies the planetary geostrophic approximation of the Boussinesq equations. We employ the novel buoyancy flow-free flow decomposition to overcome the singularity caused by the buoyancy/gravity. 


\end{abstract}

\section{Introduction}

\subsection{The Boussinesq system and the planetary geostrophic approximation}
\label{subsec:intro_pge_system}

Let $ (v,w) = (v,w)(x,y,z,t):\mathbb T^3 \times \mathbb R^+ \mapsto \mathbb R^2 \times \mathbb R $ be the horizontal and vertical velocity fields, and $ (T, p) = (T, p)(x,y,z,t): \mathbb T^3 \times \mathbb R^+ \mapsto \mathbb R\times\mathbb{R} $ be the temperature and pressure potential fields. Then the {\it non-dimensional Boussinesq equations} in the planetary geostrophic scale can be written as, for $ \varepsilon \in (0,1) $,
% The goal of this paper is to investigate the asymptotic limit $\varepsilon\to 0^+$ of solutions to the following non-dimensionalized \textit{Boussinesq equations}:
\begin{equation*}
    \tag{BE}
    \label{sys:BE}
    \begin{cases}
        \partial_t v+v\cdot\grad_h{v}+w\partial_z v+\frac{1}{\varepsilon^{\alpha}}(\grad_h{p}+f v^\perp-\Delta v)=0,                                  \\
        \partial_t w+v\cdot\grad_h{w}+w\partial_z w+\frac{1}{\varepsilon^{\alpha+2}}(\partial_z p+T)-\frac{1}{\varepsilon^{\alpha}}\Delta w=0, \\
        \grad_h{}\cdot v+\partial_z w=0,                                                                                                                       \\
        \partial_t T+v\cdot\grad_h{T}+w\partial_z T-\Delta T=Q,                                                                                \\
    \end{cases}
\end{equation*}
where $\nabla_h=(\partial_x,\partial_y)$ is the differential operator in the horizontal $(x,y)$-variable. The parameters $\varepsilon$ and $\alpha$ represent the scaling parameters for the characteristic scales; see section \ref{subsec:non-dimensionalization}, below, for more details. System \eqref{sys:BE} is used to model atmospheric flow near geostrophic balance, with large horizontal scale. Here $Q = Q(x,y,z,t): \mathbb T^3\times\mathbb R^+ \rightarrow \mathbb R$ is the given heat source for the flow, and $f$ represents the Coriolis effect, which will be taken to be $1$ in the rest of this work; see \eqref{def:f}, below. Our goal is to investigate the asymptotic dynamics of system \eqref{sys:BE} as $\varepsilon \rightarrow 0$.

Formally, sending $\varepsilon \rightarrow 0$ in system \eqref{sys:BE} leads to the \textit{planetary geostrophic equations}, given by
\begin{equation*}
    \tag{PGE}
    \label{sys:PGE}
    \begin{cases}
        \grad_h{p_p}+f v_p^\perp-\Delta v_p=0,                                  \\
        \partial_z p_p+T_p=0,                                                                \\
        \grad{}_h\cdot v_p+\partial_z w_p=0,                                                  & \\
        \partial_t T_p+v_p\cdot\grad_h{T_p}+w_p\partial_z T_p-\Delta T_p=Q, & \\
    \end{cases}
\end{equation*}
where the global well-posedness theory follows from the arguments in \cite{cao_global_2003}.
% The spacetime domain of interest is $\mathbb{T}^3\times [0,\vs{T}]$. We also require that $\alpha\in (0,2)$.
However, from \eqref{sys:BE}, one can see that the buoyancy/gravity becomes singular as $\varepsilon \rightarrow 0$ (see \eqref{eq:singular-structure}, below), providing strong instability, while the viscosity provides strong stability. The main challenge of this work is to find the balance between these two effects. 

In addition, to simplify the presentation, we assume that
\begin{equation}
    \label{asmp:homogeneous-variables}
    \int_{\mathbb{T}^3} v\,d\mathbf{x} = \int_{\mathbb{T}^3} w\,d\mathbf{x} = \int_{\mathbb{T}^3} T\,d\mathbf{x}  = \int_{\mathbb{T}^3} Q\,d\mathbf{x} = 0,
\end{equation}
and
\begin{center}
    $ v, p $ are even functions with respective to the $z$-variable, \\ and $ w, T $ are odd with respective to the $ z $-variable, 
\end{center}
    i.e.,
    \begin{equation}
        \label{asmp:symetry}
        v(-z) = v(z), \ p(-z) = p(z), \ w(-z) = - w(z),\ T(-z) = - T(z).
    \end{equation}
The direct consequence of \eqref{asmp:symetry} is that
\begin{equation}
    \label{asmp:bc}
    w\vert_{z \in \mathbb Z} = T\vert_{z \in \mathbb Z} = \partial_z v\vert_{z \in \mathbb Z} = 0. 
\end{equation}
Assumptions \eqref{asmp:homogeneous-variables}--\eqref{asmp:symetry} are invariant; that is, the same assumptions hold true for all time provided that they initially hold. 
Similar assumptions hold for $ (v_p,w_p,p_p,T_p) $. 
In particular, by restricting the solutions of \eqref{sys:BE}, \eqref{sys:PGE} in $ \lbrace 0 \leq z \leq 1 \rbrace $, one will obtain the solutions in the channel $ \mathbb T^2 \times (0,1) $, satisfying the boundary condition \eqref{asmp:bc}.

\smallskip 

% Background

% paragraph 1 -- Boussinesq equations (Newton's second law and the first law of thermodynamics for incompressible flows): why BE is important (physically), f-plane/Coriolis approximation, temperature etc.

The Boussinesq equations \eqref{sys:BE}, describing the interactive dynamics in the ocean and atmosphere, constitute the local balance of horizontal and vertical momentums, the conservation of mass for incompressible flow (incompressibility), and the first law of thermodynamics, respectively. Rotation $fv^\perp$ in $\eqref{sys:BE}_1$ is driven by Earth's Coriolis force, where $f$ is typically approximated to zeroth ($f=f_0$) or first order ($f=f_0+\beta y$), respectively referred to as the \textit{$f$-plane} and \textit{$\beta$-plane approximations} \cite{pedlosky_geophysical_1987}. Moving forward, we only consider the former by setting
\begin{equation}\label{def:f}
    f\equiv 1.
\end{equation}
The system \eqref{sys:BE} is considerably simpler than its compressible counterpart: the only contribution of density is in the buoyancy/gravity term $T\sim \rho g$ of $\eqref{sys:BE}_2$ (derived from the linear-in-$T$ equation of state and a rescaling of the temperature). Such simplification is called the \textit{Boussinesq approximation} \cite{lions_equations_1992}. In addition to $\mathbb{T}^3$, we also discuss results in the domain $\Omega:=M\times (-h,0)$, $0<h\ll 1$ constant, where, unless otherwise stated, $M\subset\mathbb{R}^2$ is smooth and $\Omega$ is equipped with the boundary conditions given in \cite{cao_global_2003}. In particular, we would like to highlight the wind-driven condition at the surface and the no heat-flux condition at the ocean floor because they play a central role in the thermodynamics of system \eqref{sys:BE}. Since the density stratification is gravitationally stable \cite{pedlosky_geophysical_1987}, the cooling of the top layer of fluid by atmospheric wind and the heating of seawater from below by Earth's surface cause hotter, lighter fluid to rise throughout the system, that is, convection occurs. The related B\'enard (wind-driven) and Rayleigh-B\'enard (heating from below) problems describe the convection more precisely, see \cite{dijkstra_dynamic_2013, fabrie_solutions_1986, fabrie_solutions_1989, foias_attractors_1987, ly_global_1999} and \cite{getling_rayleigh-benard_1998, ma_dynamic_2004, ma_rayleigh_2007, sengul_pattern_2013}, respectively.

\smallskip

% paragraph 2 -- similarity to Navier-Stokes; why is it challenging? mathematically: known results, weak solution, well-posedness is only local

Classical existence theory of the Navier-Stokes equations in $\mathbb{T}^3$ (see for instance Constantin and Foias \cite{constantin_navier-stokes_1989}) establishes the existence of global weak solutions, and the existence and uniqueness of local strong solutions for arbitrary initial data $v_0,w_0$. Uniqueness of the weak solutions is withheld due to the nonlinear advection terms preventing sufficient regularity of the time derivatives $\partial_t v,\partial _t w$. The first three equations in system \eqref{sys:BE} resemble the Navier-Stokes system, with contribution of the rotation and buoyancy/gravity. However, the effect of rotation does not create trouble in the existence theory when $f$ is constant. The added thermodynamics $\eqref{sys:BE}_4$ take the same form as the momentum equations of Navier-Stokes, meaning the temperature regularity is at least that of the velocity, so the buoyancy/gravity term in $\eqref{sys:BE}_2$ is also a nonissue in showing the existence of solutions. In fact, when $T_0\in L^\infty(\mathbb{T}^3)$, the temperature satisfies a maximum principle which can be used to show the uniqueness of the weak temperature, see Proposition \eqref{temp-est}. Therefore, with periodic boundary conditions and constant Coriolis force, the local well-posedness theory of \eqref{sys:BE} is nearly identical to that of the Navier-Stokes equations, with the existing time depending on $ \varepsilon $. The main obstacle in this work is to obtain a uniform-in-$\varepsilon $ existence theory. 

\smallskip

% paragraph 3 -- looking for simplified modeld from asymptotic expansion: hydrostatic -> PE; quasi-geostrophic approximation

We now turn our attention to some important simplified models of the Boussinesq equations. The \textit{primitive equations} (PE) are obtained under the \textit{hydrostatic approximation}. When the Froude number $\text{Fr}$ is equal to the aspect ratio $\iota$ and the other non-dimensional parameters are $\vs{O}(1)$ in \eqref{sys:be_nndmt}, the flow is near hydrostatic balance since $\iota\ll 1$, meaning it is physically justifiable to replace $\eqref{sys:be_nndmt}_2$ with $\eqref{sys:PGE}_2$ \cite{lions_equations_1992}. The \textit{quasi-geostrophic equations} (QGE) are derived for flow near %hydrostatic {and} 
geostrophic balance, where the pressure, gravity, and Coriolis force are in balance \cite{majda_introduction_2003}. The scaling from \eqref{sys:be_nndmt} is given by $\text{Fr}^2=\iota^2\text{Ro}=\iota^2\text{Ma}^2\ll 1$ and the other non-dimensional parameters are $\vs{O}(1)$. For a full derivation of the QGE (including the potential vorticity form) from the Boussinesq equations, see \cite{desjardins_derivation_1998, bardos_derivation_2024,majda_introduction_2003,bourgeois_validity_1994,embid_low_1998}.

\smallskip

% paragraphs 4 and 5 -- Mathematical analysis of PE and QG; well-posed

The existence of global weak solutions to the primitive equations in $\Omega$ was done by Lions in \cite{lions_equations_1992}, with no-slip conditions on the lateral and bottom boundaries and prescribed temperature at the ocean floor. With the same boundary conditions except no heat flux on the bottom boundary, Petcu et al. \cite{petcu_mathematical_2009} established the existence and uniqueness of local strong solutions in $\Omega$. Cao and Titi \cite{cao_global_2007} then showed the existence and uniqueness of global strong solutions in $\Omega$. See also \cite{kobelkov_existence_2006,CRMATH_2007__345_5_257_0,kukavica_regularity_2007,Cao2016a,Cao2014b,Hieber2016}.

As for the quasi-geostrophic equations, Puel and Vasseur \cite{puelGlobalWeakSolutions2015} established the existence of global weak solutions in $\mathbb{R}^2\times (0,\infty)$ with the transport boundary condition. In periodic channel domains $\mathbb{T}^2\times (0,h)$, Bennett and Kloeden \cite{bennett_periodic_1982} showed the local existence and uniqueness of strong solutions with homogeneous transport boundary conditions, and Bourgeois and Beale \cite{bourgeois_validity_1994} proved the existence of global classical solutions with homogeneous Neumann conditions. The non-uniqueness of weak solutions in $\mathbb{T}^2\times (0,2\pi)$ was later shown by Novack \cite{novack_nonuniqueness_2020}. With Ekman pumping, Desjardins and Grenier \cite{desjardins_derivation_1998} established the existence of global weak solutions in $\mathbb{T}^2\times (0,1)$, and Novack \cite{novackGlobalTimeClassical2018} established the existence and uniqueness of the global classical solution in $\mathbb{R}^2\times (0,\infty)$.

\smallskip

% paragraph 6 -- PGE is another approximation under the physical assumptions ...; Math, Cao and titi etc.

Another important simplification of the Boussinesq equations, when using the scale mentioned in section \ref{subsec:non-dimensionalization} below, is given by the planetary geostrophic equations \eqref{sys:PGE}. The planetary geostrophic equations combine the hydrostatic approximation and the geostrophic balance, where the pressure, buoyancy/gravity, Coriolis force, and viscosity are in balance in a thin domain. In \cite{samelson_mathematical_1998,samelson_remarks_2000}, the authors showed the existence of global weak solutions to \eqref{sys:PGE} in $\Omega$, and the existence and uniqueness of global strong solutions when the initial data $T_0\in L^\infty(\Omega)$ or $T_0\in H^2(\Omega)$. Cao and Titi \cite{cao_global_2003} later fully resolved the well-posedness theory by showing the uniqueness of weak solutions, and the existence and uniqueness of global strong solutions when $T_0\in H^1(\Omega)$. In all of the above results, $M$ is allowed to be a square. The well-posedness theory in $\mathbb{T}^3$ of the three systems, PE, QGE, and \eqref{sys:PGE}, follows by neglecting the boundary conditions and contributions.

\smallskip

% paragraph 7 -- Asymptotic limit: Hilbert's 6th problem (justify math model from first principal); BE is the first principal; PE is the .. limit of BE (literature); QG is the ... limit of the BE (literature, majda and embid etc.); AND Well-prepared and ill-prepared data: well-prepared data relying in formal asymptotics (Bresh) and the relative entropy method (https://www.ams.org/journals/qam/2023-81-03/S0033-569X-2023-01667-1/ and Dafermos's paper, Feireisl) ;  ill-prepared data relies on disspersive estimate (Ukai) or fast-slow decomposition (QG paper by Bardos, Liu, Titi)

With the existence theory, it is then important to justify the approximation by investigating the appropriate singular limits of the Boussinesq equations, which yield the primitive and quasi-geostrophic equations. The PE are the small aspect ratio (or small Froude number) limit of \eqref{sys:boussinesq}. Without rotation, gravity, and thermodynamics, the same is true of the PE as the limit of the Navier-Stokes equations. In $\Omega$, where $M$ is allowed to be Lipschitz, with instead no-slip conditions on the bottom boundary, Az\'erad and Guill\'en \cite{azerad_mathematical_2001} showed the convergence of a weak solution of the NSE to a weak solution of the PE. In $\mathbb{T}^3$, Li and Titi \cite{li_primitive_2019} established the convergence of the local strong solution of the NSE to the global strong solution of the PE, and Liu and Titi \cite{liu_rigorous_2024} showed the convergence of local solutions of the compressible NSE to local solutions of the compressible PE in stronger $H^s$ norms. Furthermore, Pu and Zhou \cite{pu_rigorous_2023} established the convergence of the local strong solution of the Boussinesq equations to the global strong solution of the PE in $\mathbb{T}^3$, thereby justifying the hydrostatic approximation. The QGE are more complex. When the initial data of \eqref{sys:boussinesq} is well-prepared, i.e., asymptotically close to hydrostatic and geostrophic balance, the QGE are the small Froude, small Rossby number limit of the Boussinesq equations, see \cite{embid_low_1998}. Bourgeois and Beale \cite{bourgeois_validity_1994} showed in $\mathbb{T}^2\times (0,h)$ with impermeable boundary conditions the convergence of local $C_tH_\mathbf{x}^5$ solutions of the Boussinesq equations with well-prepared (and smooth enough) initial data to local $C_tH_\mathbf{x}^6$ solutions of the QGE. See also \cite{bresch_rotating_2004}. The method of relative entropy is also used, see \cite{feireisl_rotating_2014,vasseur_review_2023}. When the initial data instead contains leading order fast gravity waves, in the whole space $\mathbb{R}^3$ dispersive and Strichartz-type estimates are available to show that the QGE are the rigorous limit of the Boussinesq equations, see \cite{charve_enhanced_2020,charve_sharper_2023}. In bounded domains, there is no way for the fast dynamics to disperse or ``escape,'' and the standard QGE, as it turns out, is not the limit of the Boussinesq equations, see Bardos et al. \cite{bardos_derivation_2024} for the correct limiting system in $\mathbb{T}^2\times (0,h)$ with impermeable boundary conditions. With Ekman pumping, Desjardins and Grenier \cite{desjardins_derivation_1998} showed in $\mathbb{T}^2\times (0,1)$ the convergence of global weak solutions of the Boussinesq equations with well-prepared (and smooth enough) initial data to local solutions of the QGE which are also smooth enough. There is one other work we would like to mention: Bresch et al. \cite{bresch_derivation_2006} established in $\mathbb{T}^3$ the convergence of local classical solutions of the PE with the well-prepared initial data $v_0,T_0=0$ to local classical solutions of a similar system to \eqref{sys:PGE}.

\smallskip

% paragraph 8 -- Our goal is to show PGE is the limit of BE. PE to PGE by Bo You (ill-prepared data, strong solution); idea of proof => BE = Buoyancy flow (geostrophic balance, no temperature) + free flow; method =  direct singular limit method, i.e., no need of relative entropy method, dispersive estimate, fast-slow wave decomposition; weak sol of BE converges to weak sol of PGE (weak sol of PGE is unique); strong sol (H^1); global and uniform in time. AND (after verify H^1) In addition, our result constructs a large class of unique global solution ($H^1$) to BE for any fixed $\varepsilon$.

The goal of this work is to justify the planetary geostrophic approximation \eqref{sys:PGE} of the Boussinesq equations \eqref{sys:BE}. We are able to establish {\it uniform-in-$\varepsilon$} estimates for solutions to system \eqref{sys:BE} without restriction on the size of the initial data; that is, our initial data is {\it general} and does not need to be well-prepared. The uniform estimates hold for all time and all sufficiently small $ \varepsilon $. In particular, we establish the global weak solution to system \eqref{sys:BE} with uniform estimate (see theorem \ref{thm:uniform-weak-sol}, below). With the uniform estimate, we are able to pass the asymptotic limit $ \varepsilon \rightarrow 0 $, and verify the planetary geostrophic approximation (see theorem \ref{thm:pge-limit}, below). Notably, while the three dimensional Boussinesq equations in general does not admit unique global weak solution \cite{noauthor_instability_2024,de_lellis_dissipative_2013,lellis_dissipative_2014}, the limit planetary geostrophic equations does admit unique global weak solution \cite{cao_global_2003}. Our result shows that, despite the possibility of non-unique weak solutions to system \eqref{sys:BE}, the asymptotic limit as $ \varepsilon\to 0 $ will always be dominated by system \eqref{sys:PGE}. Furthermore, capitalizing on the limit dynamics, we establish global unique strong solution to the Boussinesq equations \eqref{sys:BE} for general initial data but small $ \varepsilon $. This constructs a {\it large, unique solution} to the three dimensional Boussinesq equations (see theorem \ref{thm:global-strong}, below).

\smallskip

% paragraph 9 -- Summarizing the rest of the paper

The rest of this paper is organized as follows: In section \ref{subsec:results}, we summarize the main results. In section \ref{subsec:non-dimensionalization}, we formally write down the non-dimensionalization for readers' reference. In section \ref{sec:strategy}, we introduce the strategy of the proof, including the buoyancy flow-free flow decomposition to overcome the main difficulty caused by the singular buoyancy. In section \ref{sec:well-posedness}, we comment about the well-posedness of the Boussinesq equations and the buoyancy flow-free flow decomposition. In addition, we establish the basic estimates of the buoyancy flow. Section \ref{sec:est-freeflow} is devoted to the proof of global uniform-in-$\varepsilon $ estimates. The compactness and the planetary geostrophic approximation are established in section \ref{sec:proof_of_thm}. Finally, section \ref{sec:H-1-sol} establishes improved regularity and uniqueness for strong solutions.

\subsection{Main results}
\label{subsec:results}

\begin{theorem}[Global uniform stability]
    \label{thm:uniform-weak-sol} 
    Consider the initial data $ (v,w,T)\vert_{t=0} = (v_0,w_0,T_0) \in L^2(\mathbb T^3) $ and $ Q \in L^1(0,\infty;L^\infty(\mathbb T^3)) \cap L^2(0,\infty;L^2(\mathbb T^3))$. Suppose that in addition, $ T_0 \in L^\infty(\mathbb T^3)$. Then there exists $ 0< \varepsilon_0 < 1 $, depending on the initial data and $ Q $, such that for all $ \varepsilon < \varepsilon_0 $, there exists a global weak solution to system \eqref{sys:BE} with $ 0 < \alpha < 2 $, satisfying \eqref{asmp:homogeneous-variables}--\eqref{asmp:bc}, and 
    \begin{equation}
        \label{thmest:weak-global-sol}
        \sup_{0\leq t < \infty} \bigl\lbrace \norm[2]{v(t),\varepsilon w(t),T(t)}^2 + \norm[\infty]{T(t)}^2 \bigr\rbrace + \int_0^\infty \bigl\lbrace \norm[H^1]{v(t),\varepsilon w(t),T(t)}^2 + \norm[2]{w(t)}^2 \bigr\rbrace \,dt < \infty,
    \end{equation}
    uniformly-in-$ \varepsilon $. 
\end{theorem}

\begin{proof}
    The proof of theorem \ref{thm:uniform-weak-sol} follows from sections \ref{sec:well-posedness},  \ref{sec:est-freeflow}, and \ref{subsec:approximation}. 
\end{proof}

\begin{theorem}[Planetary geostrophic approximation]
    \label{thm:pge-limit}
    Fix $ \mathcal T \in (0,\infty) $. 
    Under the same conditions as in theorem \ref{thm:uniform-weak-sol}, upon selecting a subsequence $ \varepsilon \rightarrow 0 $, 
    \begin{equation} (v,w,T) \rightarrow (v_p,w_p,T_p) \quad \text{in the proper topology (see \eqref{limit:weak-1}--\eqref{limit:strong-4}, below)}, \end{equation}
    with 
    \begin{equation}
    v_p \in L^\infty(0,\mathcal T;L^2(\mathbb T^3)) \cap L^2(0,\mathcal T;H^2(\mathbb T^3)), \quad w_p, T_p \in L^\infty(0,\mathcal T;L^2(\mathbb T^3)) \cap L^2(0,\mathcal T;H^1(\mathbb T^3)), 
    \end{equation}
    such that $ (v_p, w_p, T_p) $ is a global weak solution to the planetary geostrophic equations \eqref{sys:PGE}.
\end{theorem}

\begin{proof}
    The proof of theorem \ref{thm:pge-limit} follows from section \ref{subsec:compactness}.
\end{proof}

\begin{theorem}[Global strong solution]
    \label{thm:global-strong} 
    In addition to the conditions of theorem \ref{thm:uniform-weak-sol}, consider the initial data $ (v_0,w_0,T_0) \in H^1(\mathbb T^3) $. Then there exists $ 0<\varepsilon_1 < \varepsilon_0 $, depending on the initial data and $ Q $, such that for all $ \varepsilon < \varepsilon_1 $, there exists a unique global strong solution to system \eqref{sys:BE} with $ 0 < \alpha < 2 $, satisfying 
    \begin{equation}
        \label{thmest:strong-global-sol}
        \sup_{0\leq t < \infty} \bigl\lbrace \norm[H^1]{v(t), \varepsilon w(t), T(t)}^2 + \norm[2]{w(t)}^2 \bigr\rbrace + \int_0^\infty \bigl\lbrace \norm[H^2]{v(t),\varepsilon w(t),T(t)}^2 + \norm[H^1]{w(t)}^2 \bigr\rbrace \,dt < \infty, 
    \end{equation}
    uniformly-in-$\varepsilon $. 
\end{theorem}

\begin{proof}
    The proof of theorem \ref{thm:global-strong} follows from section \ref{sec:H-1-sol}.
\end{proof}

\subsection{Non-dimensionalization}
\label{subsec:non-dimensionalization}

For the sake of completeness,
we present here the formal derivation of the singular system \eqref{sys:BE}; that is, the non-dimensionalization of the Boussinesq equations. Recall them in their (unscaled) dimensional form:
\begin{equation}
    \label{sys:boussinesq}
    \begin{cases}
        \partial_t v+v\cdot\grad_h{v}+w\partial_z v+\grad_h{p}+v^\perp-\nu_h\Delta_h v-\nu_v\partial_{zz}v=0, \\
        \partial_t w+v\cdot\grad_h{w}+w\partial_z w+\partial_z p+T-\nu_h\Delta_h w-\nu_v\partial_{zz}w=0,    \\
        \grad_h{}\cdot v+\partial_z w=0,                                                                   \\
        \partial_t T+v\cdot\grad_h{T}+w\partial_z T-\kappa_h\Delta_h T-\kappa_v\partial_{zz} T=Q,            \\
    \end{cases}
\end{equation}
where $\Delta_h=\partial_{xx}+\partial_{yy}$ is the Laplacian in the horizontal $(x,y)$-variable. By writing all (independent and dependent) variables in the form of multiplication of the typical values with the non-dimensional variables,
one can write
\begin{gather*}
    (x,y,z,t)=(Lx',Ly',Hz',\tau t'),\quad (v,w)=(Vv',Ww'),\quad p=Pp',\quad T=\tilde{T}T',\quad Q=\tilde{Q}Q',
\end{gather*}
where the $(\cdot)'$ variables are dimensionless and $\vs{O}(1)$. Abusing the notation and dropping the primes, system \eqref{sys:boussinesq} can be written as
\begin{equation}
    \label{sys:be_nndmt}
    \begin{cases}
        \partial_t v+v\cdot\grad_h{v}+w\partial_z v+\frac{1}{\text{Ma}^2}\grad_h{p}+\frac{1}{\text{Ro}}v^\perp-\frac{1}{\text{Re}_h}\Delta_h v-\frac{1}{\text{Re}_v}\partial_{zz}v=0,                  \\
        \partial_t w+v\cdot\grad_h{w}+w\partial_z w+\frac{1}{\text{Ma}^2}\frac{1}{\iota^2}\partial_z p+\frac{1}{\text{Fr}^2}T-\frac{1}{\text{Re}_h}\Delta_h w-\frac{1}{\text{Re}_v}\partial_{zz}w=0, \\
        \grad_h{}\cdot v+\partial_z w=0,                                                                                                                                                                     \\
        \partial_t T+v\cdot\grad_h{T}+w\partial_z T-\frac{1}{\text{Pe}_h}\Delta_h T-\frac{1}{\text{Pe}_v}\partial_{zz}T=\frac{\tilde{Q}L}{\tilde{T}V}Q.                                        \\
    \end{cases}
\end{equation}
Here we have set $\frac{1}{\tau}=\frac{V}{L}=\frac{W}{H}$ and
\begin{equation}
    \label{characteristic-number}
    \text{Ma}:=\frac{V}{\sqrt{P}},\quad\text{Ro}:=\frac{V}{L},\quad\text{Fr}:=\frac{W}{\sqrt{\tilde{T}H}},\quad\text{Re}_h:=\frac{LV}{\nu_h},\quad\text{Re}_v:=\frac{HW}{\nu_v},\quad \text{Pe}_h:=\frac{LV}{\kappa_h},\quad\text{Pe}_v:=\frac{HW}{\kappa_v},
\end{equation}
where $\text{Ma}$ is the Mach number, $\text{Ro}$ the Rossby number, $\text{Fr}$ the Froude number, and $\text{Re}_i$ and $\text{Pe}_i$ for $i\in \{h,v\}$ the horizontal and vertical Reynolds and Péclet numbers, respectively. We also define the {aspect ratio} $\iota:=\frac{H}{L}=\frac{W}{V}$. Taking
\begin{equation}
    \label{def:scale}
    \iota=\varepsilon,\quad \text{Fr}^2=\varepsilon^{\alpha+2},\quad \text{Ma}^2=\text{Ro}=\text{Re}_h=\text{Re}_v=\varepsilon^\alpha,\quad\text{and}\quad\text{Pe}_h=\text{Pe}_v=\frac{\tilde{T}V}{\tilde{Q}L}=1
\end{equation}
in system \eqref{sys:be_nndmt}
leads to \eqref{sys:BE}.

\subsection{Functional spaces and notation}
In the following, we denote by
\begin{align*}
    \norm[p]{f}:=\left(\int_{\mathbb{T}^3}|f|^p\,d\mathbf{x}\right)^\frac{1}{p}
\end{align*}
the $L^p(\mathbb{T}^3)$ norm for every $f\in L^p(\mathbb{T}^3)$, by
\begin{align*}
    (f,g):=\int_{\mathbb{T}^3}fg\,d\mathbf{x}
\end{align*}
the $L^2(\mathbb{T}^3)$ inner product for every $f,g\in L^2(\mathbb{T}^3)$, and by
\begin{align*}
    \norm[H^s]{f}:=\left(\sum_{|\alpha|\leq s} \norm[2]{\partial^\alpha f}\right)^\frac{1}{2},\quad s\geq 0\text{ and }\alpha\in \mathbb{Z}^3_{\geq 0}\text{ a multi-index},
\end{align*}
the Sobolev $H^s(\mathbb{T}^3)$ norm for every $f\in H^s(\mathbb{T}^3)$, with the convention that $H^0(\mathbb{T}^3)=L^2(\mathbb{T}^3)$. We let $H^{-s}(\mathbb{T}^3)$ be the dual space of $\dot{H}^s(\mathbb{T}^3):=\{f\in H^s(\mathbb{T}^3) : \int_{\mathbb{T}^3} f\,d\mathbf{x}=0\}$, with operator norm
\begin{align*}
    \norm[H^{-s}]{\phi}:=\sup_{\Vert{\grad^s{f}}\Vert_{2}=1}\phi(f)
\end{align*}
for every $\phi\in H^{-s}(\mathbb{T}^3)$. We denote the dual action between $ H^{-s} $ and $ \dot H^s $ by
\begin{align*}
    \langle \phi,f\rangle:=\phi(f).
\end{align*}
We also utilize the Bochner spaces $L^p(I;X)$, where $I\subset [0,\infty)$ is an interval and $X$ is some normed space over $\mathbb{T}^3$, with norm
\begin{align*}
    \norm[L^p(I;X)]{f}:=\left(\int_I\,\norm[X]{f(t)}^p\,dt\right)^\frac{1}{p}
\end{align*}
for every $f\in L^p(I;X)$. Similar notations are used for functional spaces over $ \mathbb T^2 $.
% When the domain is something other than $\mathbb{T}^3$, we specify it in the subscript of the norm, e.g.,
% \begin{align*}
%     \norm[L^p(\mathbb{T}^2)]{f(z)}:=\left(\int_{\mathbb{T}^2}|f(z)|^p\,dxdy\right)^\frac{1}{p}
% \end{align*}
% for every $f(z)\in L^p(\mathbb{T}^2)$. 
We abuse the notation by using the same functional spaces, e.g., $L^2(\mathbb{T}^3)$, to include both scalar- and vector-valued functions. %, and oftentimes we write $\Vert{f}\Vert_{H^s}$ instead of $\Vert{f(t)}\Vert_{H^s}$ when $f$ is time-dependent.

Lastly, we use the notation
\begin{equation*}
    A \lesssim B
\end{equation*}
for any $ A $ and $ B $
to represent 
\begin{equation*}
    A \leq C B 
\end{equation*}
for some generic constant $ C \in (0,\infty) $, different from line to line, uniform-in-$\varepsilon$.

% \begin{align*}
%     \begin{cases}
%         \frac{V}{\tau}\partial_t v+\frac{V^2}{L}v\cdot\grad{v}+\frac{WV}{H}w\partial_z v+\frac{P}{L}\grad{p}+FVv^\perp-\frac{\tilde{\nu}_h V}{L^2}\Delta v-\frac{\tilde{\nu}_v V}{H^2}\partial_{zz}v=0, &\\
%         \frac{W}{\tau}\partial_t w+\frac{VW}{L}v\cdot\grad{w}+\frac{W^2}{H}w\partial_z w+\frac{P}{H}\partial_z p+\tilde{T}T-\frac{\tilde{\nu}_h W}{L^2}\Delta w-\frac{\tilde{\nu}_v W}{H^2}\partial_{zz}w=0, &\\
%         \frac{V}{L}\grad{}\cdot v+\frac{W}{H}\partial_z w=0, &\\
%         \frac{\tilde{T}}{\tau}\partial_t T+\frac{V\tilde{T}}{L}v\cdot\grad{T}+\frac{W\tilde{T}}{H}w\partial_z T-\frac{\tilde{\kappa}_h\tilde{T}}{L^2}\Delta T-\frac{\tilde{\kappa}_v\tilde{T}}{H^2}\partial_{zz}T=\tilde{Q}Q. &\\
%     \end{cases}
% \end{align*}
% Now consider the (non-stationary) scale $\frac{1}{\tau}=\frac{V}{L}=\frac{W}{H}$ and set
% \begin{gather*}
%     \text{Ma}:=\frac{V}{\sqrt{P}},\quad\text{Ro}:=\frac{V}{FL},\quad\text{Fr}:=\frac{W}{\sqrt{\tilde{T}H}},\quad\text{Re}_h:=\frac{LV}{\tilde{\nu}_h},\quad\text{Re}_v:=\frac{HW}{\tilde{\nu}_v},\quad \text{Pe}_h:=\frac{LV}{\tilde{\kappa}_h},\quad\text{Pe}_v:=\frac{HW}{\tilde{\kappa}_v},
% \end{gather*}
% where $\text{Ma}$ is the Mach number, $\text{Ro}$ the Rossby number, $\text{Fr}$ the Froude number, and $\text{Re}_i$ and $\text{Pe}_i$ for $i\in \{h,v\}$ the horizontal and vertical Reynolds and Péclet numbers, respectively. We also define the \textit{aspect ratio} $\iota:=\frac{H}{L}=\frac{W}{V}$. Then the dimensionless system can be written as

% \section{Preliminary}

\section{Strategy and the buoyancy flow-free flow decomposition}
\label{sec:strategy}

The key step in rigorously justifying the asymptotic limit $ \varepsilon \rightarrow 0 $ in system \eqref{sys:BE}, leading to system \eqref{sys:PGE}, is obtaining a uniform-in-$\varepsilon$ estimate. However, due to the singular limit nature of system \eqref{sys:BE}, this is impossible in the current form. Indeed, a natural energy estimate of the kinetic energy, obtained by formally taking the $L^2(\mathbb{T}^3)$ inner product of $\eqref{sys:BE}_1$ and $\eqref{sys:BE}_2$ with $ v $ and $ \varepsilon^2 w $, respectively, and adding the resultants together leads to
\begin{equation}
    \label{eq:singular-structure}
    \frac{1}{2}\frac{d}{dt}\norm[2]{v,\varepsilon w}^2 + \norm[2]{\varepsilon^{-\alpha/2}\grad{v},\varepsilon^{1-\alpha/2}\grad{w}}^2 = -\varepsilon^{-\alpha}\int_{\mathbb{T}^3} T w \,d\mathbf{x},
\end{equation}
where the right hand side, caused by the buoyancy/gravity, is singular as $ \varepsilon \rightarrow 0 $.

In principle, such singular behavior is caused by the underlining planetary geostrophic dynamics, dominated by the limit system \eqref{sys:PGE}. Thus the dissipation on the left hand side of \eqref{eq:singular-structure} is an over estimate, leading to the singular right hand side. 

To overcome this difficulty, we introduce the following  decomposition to capture the limit, i.e., the planetary geostrophic dynamics of system \eqref{sys:BE}:
\begin{center}
    The Boussinesq flow $ = $ Buoyancy flow $ + $ Free flow.
\end{center}

To be more precise, we introduce the decomposition of $(v,w,p)$ into the following form:
\begin{align}\label{decomp}
    (v,w,p)=(v_T,w_T,p_T)+(v_F,w_F,p_F),
\end{align}
where $(v_T,w_T,p_T)$ is the flow ``forced'' by the temperature through buoyancy/gravity, referred to as the \textit{buoyancy flow}, satisfying
\begin{equation*}
    \tag{BF}
    \label{sys:BF}
    \begin{cases}
        \grad_h{p_T}+v_T^\perp-\Delta v_T=0, & \\
        \partial_z p_T+T=0,                                 & \\
        \grad_h{}\cdot v_T+\partial_z w_T=0.                  & \\
    \end{cases}
\end{equation*}
Then one can derive from the \eqref{sys:BE} and \eqref{sys:BF} that the \textit{free flow} $(v_F,w_F,p_F,T)$ satisfies
\begin{equation*}
    \tag{FF}
    \label{sys:FF}
    \begin{cases}
        \partial_t v_F+v_F\cdot\grad_h{v_F}+w_F\partial_z v_F+\frac{1}{\varepsilon^\alpha}(\grad_h{p_F}+v_F^\perp-\Delta v_F)                          \\
        \hspace{0.8cm}=-\partial_t v_T-v\cdot\grad_h{v_T}-v_T\cdot\grad_h{v_F}-w\partial_z v_T-w_T\partial_z v_F,                                                     & \\
        \partial_t w_F+v_F\cdot\grad_h{w_F}+w_F\partial_z w_F+\frac{1}{\varepsilon^{\alpha+2}}\partial_z p_F-\frac{1}{\varepsilon^\alpha}\Delta w & \\
        \hspace{0.8cm}=-\partial_t w_T-v\cdot\grad_h{w_T}-v_T\cdot\grad_h{w_F}-w\partial_z w_T-w_T\partial_z w_F,                                                     & \\
        \grad_h{}\cdot v_F+\partial_z w_F=0,                                                                                                                        & \\
        \partial_t T+v_F\cdot\grad_h{T}+w_F\partial_z T-\Delta T=Q-v_T\cdot\grad_h{T}-w_T\partial_z T,
    \end{cases}
\end{equation*}
where certain terms have not been expanded to save space. %, or because their formal expansions are unnecessary. 
We refer to \eqref{sys:BF}-\eqref{sys:FF} together as the \textit{intermediate system} of the buoyancy flow-free flow decomposition.
Notably, without loss of generality, we assume that \eqref{asmp:homogeneous-variables}--\eqref{asmp:bc} holds for both the buoyancy flow and the free flow.

Consider the following initial data for the intermediate system \eqref{sys:BF}-\eqref{sys:FF}:
\begin{equation}
    \label{initial:intermediate_sys}
    (v_F,w_F,T)\vert_{t=0} = (v_{F,0},w_{F,0},T_0).
\end{equation}
Let the initial data satisfy the following conditions:
\begin{equation}
    \label{initial:space}
    \begin{gathered}
        T_0 \in L^\infty(\mathbb T^3) \cap L^2(\mathbb T^3), \\
        v_{F,0}, w_{F,0} \in L^2(\mathbb T^3), \qquad \grad_h{v_{F,0}} + \partial_z w_{F,0} = 0 \quad \text{in the sense of distribution}.
    \end{gathered}
\end{equation}
The heat source $ Q $ is given, and satisfies
\begin{equation}
    \label{source:q}
    Q \in L^1(0,\infty; L^\infty(\mathbb T^3))\cap L^2(0,\infty;L^2(\mathbb T^3)).
\end{equation}

Our key step is then to obtain uniform-in-$\varepsilon$ estimates of the buoyancy flow, dominated by system \eqref{sys:BF}, and the free flow, dominated by system \eqref{sys:FF}. Notice that, system \eqref{sys:BF} is nothing but a linear elliptic system for $ (v_T, w_T, p_T) $ for any given $ T $. This captures the structure of the planetary geostrophic system \eqref{sys:PGE}. With $ (v_T, w_T, p_T) $ given, system \eqref{sys:FF} is a not-as-singular system for $ (v_F,w_F, p_F, T) $ with some source terms from the buoyancy flow. This particular decomposition captures the strong dissipation of the free flow, while preserving the limit dynamics of the planetary geostrophic flow.

\section{Existence of solutions and well-posedness for any fixed $\varepsilon>0$, elliptic estimates, and the temperature estimates}
\label{sec:well-posedness}

The intermediate system \eqref{sys:BF}--\eqref{sys:FF} is just a rewriting of the original Boussinesq system \eqref{sys:BE}. Therefore, the existence of global weak solutions, and the well-posedness of local strong solutions follow directly from the classical theory, with the estimates and/or existing time depending on $ \varepsilon $, for any fixed $\varepsilon>0$; see \cite{constantin_navier-stokes_1989}.

In order to investigate the limit $\varepsilon\to 0$ in system \eqref{sys:BE}, the main objective is to obtain uniform-in-$\varepsilon$ estimates of $(v,w,p,T)$. However, thanks to our decomposition \eqref{decomp}, this is equivalent to obtaining uniform-in-$\varepsilon$ estimates of the buoyancy flow $(v_T,w_T,p_T)$ and the free flow $(v_F,w_F,p_F,T)$. Moreover, notice that the buoyancy flow satisfies the linear elliptic problem \eqref{sys:BF} for any given $T$. This defines a map, referred to as the \textit{buoyancy map},
\begin{equation}
    \label{def:buoyancy_map}
    T\mapsto (v_T,w_T,p_T).
\end{equation}

Our first important ingredient of this work is the following property of the buoyancy map \eqref{def:buoyancy_map}:
\begin{proposition}[Buoyancy map]\label{bf-est}
    The buoyancy map \eqref{def:buoyancy_map} is well-defined for $ T $ in the proper space. In addition, the following estimates hold:
    %The following estimates hold for the solution of \eqref{sys:BF},
    \begin{enumerate}[label=(\roman*)]
        \item For every integer $s\geq -1$ such that $T\in H^s(\mathbb{T}^3)$, one has that
        \begin{equation}
                  \label{est:bf-101}
                  \norm[H^s]{w_T}\lesssim \norm[H^{s+1}]{v_T}\lesssim \norm[H^s]{T}, \end{equation}
                  and 
                  \begin{equation}
                  \label{est:bf-101-2}
                  \norm[H^s]{\partial_t w_T}\lesssim \norm[H^{s+1}]{\partial_t v_T}\lesssim \norm[H^s]{\partial_t T}; \end{equation}
        \item For every $ T \in L^2(\mathbb T^3) $ and any $ s \geq 0 $, one has that
              \begin{equation}
                  \label{est:bf-102}
                  \norm[H^{s+1}]{\partial_z v_T}\lesssim \norm[H^s]{T}.
              \end{equation}
        % \item The following estimate holds \begin{equation}
        %           \label{est:bf-101-2}
        %           \norm[H^{-1}_x]{\partial_t w_T}\lesssim \norm[2]{\partial_t v_T}\lesssim \norm[H^{-1}_x]{\partial_t T}.
        %       \end{equation}
        %       % \item $\norm[H^{-1}_x]{w_T}\lesssim\norm[2]{v_T}\lesssim \norm[H^{-1}_x]{T}$.
    \end{enumerate}
\end{proposition}
\begin{proof}
    \begin{enumerate}[label=(\roman*)]
        \item {When $s=0$}, we take the $L^2(\mathbb{T}^3)$ inner product of $ \eqref{sys:BF}_1 $ and $\eqref{sys:BF}_2$ with $v_T$ and $w_T$, respectively, add the resultants together, and integrate by parts using $ \eqref{sys:BF}_3 $,
              \begin{equation}
                  \label{est:bf-001}
                  \norm[2]{\grad{v_T}}^2=-(T,w_T)\leq\norm[2]{T}\norm[2]{w_T},
              \end{equation}
              where we have applied H\"older's inequality.
              $\eqref{sys:BF}_3$, together with \eqref{asmp:bc}, also implies
              \begin{align}
                  \label{wt-bound-vt-0}
                  w_T=-\int_0^z \grad_h{}\cdot v_T(\zeta)\,d\zeta, \\
                  \intertext{and thus}\label{wt-bound-vt}
                  \norm[2]{w_T}\leq \norm[2]{\grad{v_T}}.
              \end{align}
              Applying Young's inequality in \eqref{est:bf-001} then implies
              \begin{equation}
                  \label{est:bf-002}
                  \norm[2]{w_T}\leq \norm[2]{\grad{v_T}}\lesssim \norm[2]{T}.
              \end{equation}
              Thanks to \eqref{asmp:homogeneous-variables}, applying the Poincaré inequality in \eqref{est:bf-002} yields the desired bounds.

              \smallskip

              When $s \geq 1$, \eqref{est:bf-101} is a consequence of the linear structure of the equations. Indeed, taking the $L^2(\mathbb{T}^3)$ inner product of $\eqref{sys:BF}_1 $ and $ \eqref{sys:BF}_2 $ with $(-\Delta)^{s}v_T$ and $(-\Delta)^{s}w_T$, respectively, and applying similar arguments as above lead to the desired bounds.

              \smallskip

              When $s=-1$, we alternatively take the $L^2(\mathbb{T}^3)$ inner product of $ \eqref{sys:BF}_1 $ and $ \eqref{sys:BF}_2$ with $(-\Delta)^{-1}v_T$ and $(-\Delta)^{-1}w_T$, respectively, add the resultants together, and integrate by parts using $ \eqref{sys:BF}_3 $,
              \begin{equation}
                  \label{est:bf-003}
                  \norm[2]{v_T}^2=-(T,(-\Delta)^{-1}w_T)\leq \norm[H^{-1}]{T}\norm[H^1]{(-\Delta)^{-1}w_T} \leq\norm[H^{-1}]{T}\norm[H^{-1}]{w_T},
              \end{equation}
              which implies the desired bounds after utilizing \eqref{wt-bound-vt-0}.

              The proof of \eqref{est:bf-101-2} follows immediately after noticing that the time derivative of $ v_T, w_T, T $ satisfy the same equations as $ v_T,w_T,T $. 

        \item For $s = 0 $, taking $\partial_z$ of $\eqref{sys:BF}_1$ yields, thanks to $ \eqref{sys:BF}_2 $,
              \begin{equation}
                  \label{est:bf-004}
                  -\Delta\partial_z v_T=-\grad_h{\partial_z p_T}-(\partial_z v_T)^\perp=\grad_h{T}- (\partial_z v_T)^\perp.
              \end{equation}
              Now taking the $L^2(\mathbb{T}^3)$ inner product of \eqref{est:bf-004} with $\partial_z v_T$, integrating by parts, and applying H\"older's inequality implies
              \begin{equation}
                  \label{est:bf-005}
                  \norm[2]{\grad{\partial_z v_T}}^2 =(\grad_h{T},\partial_z v_T)=-(T,\grad_h{}\cdot \partial_z v_T)\leq \norm[2]{T}\norm[2]{\grad{\partial_z v_T}},
              \end{equation}
              from which the desired bound follows after using the Poincaré inequality.

              The case when $ s > 0 $ follows similarly. 
        % \item The proof is practically identical to the $s=-1$ case of (i) after taking $\partial_t$ of $\eqref{sys:BF} $.
    \end{enumerate}
\end{proof}

In addition to Proposition \eqref{bf-est}, we also provide the following estimate for $ T $:
\begin{proposition}[Temperature estimates]\label{temp-est}
    Assume that $ T $ solves $ \eqref{sys:BE}_4 $, or equivalently  $ \eqref{sys:FF}_4 $, with given $ Q $ and $ v_F,w_F $ in the proper space.
    The following uniform-in-$\varepsilon$ estimates hold:
    \begin{enumerate}[label=(\roman*)]
        \item For all $ \vs{T} \in (0,\infty) $, one has that
              \begin{equation}
                  \label{est:tmp-101}
                  \sup\limits_{t\in [0,\vs{T}]}\norm[\infty]{T(t)}\leq \norm[\infty]{T_0}+\int_0^\vs{T}\norm[\infty]{Q(t)}\,dt;
              \end{equation}
        \item For all $ \mathcal T \in (0,\infty) $, one has that
              \begin{equation}
                  \label{est:tmp-102}
                  \sup\limits_{t\in [0,\vs{T}]}\norm[2]{T(t)}^2+\int_0^\vs{T} \norm[2]{\grad{T}(t)}^2\,dt\leq \norm[2]{T_0}^2+\int_0^\vs{T}\norm[2]{Q(t)}^2\,dt;
              \end{equation}
        \item One has that
              % \begin{equation}
              %           \label{est:tmp-103}
              %           \norm[H^{-1}_x]{\partial_t T}\lesssim \norm[2]{\grad_3{T}}+\norm[\infty]{T}\norm[2]{\grad_3{T}}+\norm[6]{T}\norm[3]{v_F,w_F},
              %       \end{equation}
              %       and
              \begin{equation}
                  \label{est:tmp-103}
                  \norm[H^{-1}]{\partial_t T}\lesssim \norm[2]{Q}+\norm[2]{\grad{T}}+\norm[\infty]{T}\norm[2]{\grad{T}}+\norm[\infty]{T}\norm[2]{v_F,w_F}.
              \end{equation}
    \end{enumerate}
\end{proposition}
\begin{proof}
    \begin{enumerate}[label=(\roman*)]
        \item For all $q>1$, taking the $L^2(\mathbb{T}^3)$ inner product of $ \eqref{sys:BE}_4 $ with $q|T|^{q-2} T$, integrating by parts, and using $ \eqref{sys:BE}_3 $ leads to
              \begin{equation}
                  \label{est:tmp-001}
                  \begin{gathered}
                      \frac{d}{dt}\norm[q]{T}^q+q(q-1)\int_{\mathbb{T}^3}|T|^{q-2}|\grad{T}|^2\,d\mathbf{x} \\ =q\int_{\mathbb{T}^3}|T|^{q-2}TQ\,d\mathbf{x}
                      \leq q\norm[q]{T}^{q-1}\norm[q]{Q}
                      \leq q\norm[q]{T}^{q-1}\norm[\infty]{Q}.
                  \end{gathered}
              \end{equation}
              Hence, one has that
              \begin{equation}
                  \label{est:tmp-002}
                  q\norm[q]{T}^{q-1}\frac{d}{dt}\norm[q]{T}\leq q\norm[q]{T}^{q-1}\norm[\infty]{Q},
              \end{equation}
              which implies
              \begin{equation}\label{est:tmp-003}
                  \norm[q]{T(t)}\leq \norm[q]{T_0}+\int_0^t\norm[\infty]{Q(s)}\,ds.
              \end{equation}
              Sending $q\to \infty$ and taking the supremum over $t\in [0,\vs{T}]$ yield the desired bound.
        \item
              % Taking the $L^2_x$ inner product of $ \eqref{sys:FF}_4 $ with $T$, integrating by parts, and using $ \eqref{sys:BE}_3 $ imply
              Consider \eqref{est:tmp-001} with $ q = 2 $, and apply H\"older's inequality, the Poincar\'e inequality thanks to \eqref{asmp:bc}, and Young's inequality. One has that
              \begin{equation}
                  \label{est:tmp-004}
                  \begin{gathered}
                      \frac{d}{dt}\norm[2]{T}^2+ 2 \norm[2]{\grad{T}}^2 = 2 (Q,T)\leq 2 \norm[2]{Q}\norm[2]{T} \\
                      \leq 2 \norm[2]{Q}\norm[2]{\grad{T}} \leq \norm[2]{Q}^2+\norm[2]{\grad{T}}^2.
                  \end{gathered}
              \end{equation}
              Integrating \eqref{est:tmp-004} in time yields the desired bound.
        \item Thanks to $\eqref{sys:BE}_3$, $ \eqref{sys:BE}_4 $ can be written as
              \begin{equation}
                  \label{est:tmp-005}
                  \partial_t T=Q+\grad{}\cdot (\grad{T}-T(v,w)).
              \end{equation}
              Hence, one has that
              \begin{equation}
                  \label{est:tmp-006}
                  \begin{gathered}
                      \norm[H^{-1}]{\partial_t T}\lesssim \norm[H^{-1}]{Q}+\norm[2]{\grad{T}-T(v,w)}
                      \lesssim \norm[2]{Q}+\norm[2]{\grad{T}}+\norm[2]{T(v,w)}\\
                      \leq \norm[2]{Q}+\norm[2]{\grad{T}}+\norm[\infty]{T}\norm[2]{v_T,w_T}+\norm[\infty]{T}\norm[2]{v_F,w_F} \\
                      \overset{\eqref{est:bf-101}}{\lesssim} \norm[2]{Q}+\norm[2]{\grad{T}}+\norm[\infty]{T}\norm[2]{\grad{T}}+\norm[\infty]{T}\norm[2]{v_F, w_F},
                  \end{gathered}
              \end{equation}
              where we have also applied Poincar\'e's inequality.
    \end{enumerate}
\end{proof}

\section{Estimates of the free flow}
\label{sec:est-freeflow}

What remains is to establish the estimates for the free flow, dominated by system \eqref{sys:FF}. Throughout this section, all calculations are formal. We will outline the justification in section \ref{subsec:approximation}, below.

Taking the $L^2(\mathbb{T}^3)$ inner product of $ \eqref{sys:FF}_1 $ and $ \eqref{sys:FF}_2 $ with $v_F$ and $\varepsilon^2 w_F$, respectively, integrating by parts, and adding the resultants together, thanks to $ \eqref{sys:FF}_3 $, leads to
% \begin{align*}
%     \begin{cases}
%         \frac{1}{2}\frac{d}{dt}\norm[2]{v_F}^2-\varepsilon^{-\alpha}(p_F,\grad{}\cdot v_F)+\varepsilon^{-\alpha}(\nu_h\norm[2]{\grad{v_F}}^2+\nu_v\norm[2]{\partial_z v_F}^2) &\\
%         \hspace{0.8cm}= -(\partial_t v_T,v_F)-(v\cdot\grad{v_T},v_F)-(w\partial_z v_T,v_F), &\\
%         \varepsilon^2\frac{1}{2}\frac{d}{dt}\norm[2]{w_F}^2-\varepsilon^{-\alpha}(p_F,\partial_z w_F)+\varepsilon^{2-\alpha}(\nu_h\norm[2]{\grad{w_F}}^2+\nu_v\norm[2]{\partial_z w_F}^2) &\\
%         \hspace{0.8cm}=-\varepsilon^2(\partial_t w_T,w_F)-\varepsilon^2(v\cdot\grad{w_T},w_F)-\varepsilon^2(w\partial_z w_T,w_F) &\\
%         \hspace{1.2cm}-\varepsilon^{2-\alpha}(\nu_h(\grad{w_T},\grad{w_F})+\nu_v(\partial_z w_T,\partial_z w_F)). &\\
%     \end{cases}
% \end{align*}
% Adding the equations together and using (FF3),
\begin{equation}
    \label{est:ff-000}
    \begin{gathered}
        \frac{1}{2}\frac{d}{dt}\norm[2]{v_F,\varepsilon w_F}^2 +\varepsilon^{-\alpha}\norm[2]{\grad{v_F}}^2+\varepsilon^{2-\alpha}\norm[2]{\grad{w_F}}^2
        = -(\partial_t v_T,v_F)-\varepsilon^2(\partial_t w_T,w_F) \\
        -(v\cdot\grad_h{v_T},v_F)-(w\partial_z v_T,v_F)
        -\varepsilon^2(v\cdot\grad_h{w_T},w_F)-\varepsilon^2(w\partial_z w_T,w_F)
        \\ -\varepsilon^{2-\alpha}(\grad{w_T},\grad{w_F})
        %     .
        % \leq |(\partial_t v_T,v_F)|+\varepsilon^2|(\partial_t w_T,w_F)|+|(v\cdot\grad{v_T},v_F)|+|(w\partial_z v_T,v_F)|\\
        % \hspace{0.8cm}+\varepsilon^2|(v\cdot\grad{w_T},w_F)|+\varepsilon^2|(w\partial_z w_T,w_F)|\\
        % \hspace{0.8cm}+\varepsilon^{2-\alpha}(\nu_h|(\grad{w_T},\grad{w_F})|+\nu_v|(\partial_z w_T,\partial_z w_F)|)\\
        =:\sum_{j=1}^7 I_j.
    \end{gathered}
\end{equation}

\paragraph{Estimate of $I_1$.} Applying H\"older's inequality and Poincar\'e's inequality, one has that
\begin{equation}
    \label{est:ff-001}
    \begin{aligned}
        I_1 & \leq |(\partial_t v_T,v_F)|\leq \norm[2]{\partial_t v_T}\norm[2]{v_F}
        \overset{\eqref{est:bf-101-2}}{\lesssim} \norm[H^{-1}]{\partial_t T}\norm[2]{v_F}\\
        & \overset{\eqref{est:tmp-103}}{\lesssim} \norm[2]{Q}\norm[2]{\grad{v_F}}+\norm[2]{\grad{T}}\norm[2]{\grad{v_F}}+\norm[\infty]{T}\norm[2]{\grad{T}}\norm[2]{\grad{v_F}}+\norm[\infty]{T} \underbrace{\norm[2]{v_F,w_F}}_{\mathclap{\lesssim \norm[2]{\grad{v_F}}}} \norm[2]{v_F}.
    \end{aligned}
\end{equation}
where we have applied, thanks to $ \eqref{sys:FF}_{3} $ and Minkowski's inequality, that
\begin{equation}
    \label{est:ff-incomp}
    \norm[2]{w_F} = \norm[2]{\int_0^z \grad_h{}\cdot v_F(\zeta) \,d\zeta} \leq \norm[2]{\grad{v_F}}.
\end{equation}
Therefore, applying Young's inequality in \eqref{est:ff-001} implies that, for some constant $C\in (0,\infty)$,
\begin{equation}
    \label{est:ff-002}
    \begin{gathered}
        I_1 \leq \frac{1}{100}\varepsilon^{-\alpha}\norm[2]{\grad{v_F}}^2 + \varepsilon^\alpha C\norm[2]{Q}^2 +\varepsilon^\alpha C\norm[2]{\grad{T}}^2+\varepsilon^\alpha C\norm[\infty]{T}^2\norm[2]{\grad{T}}^2\\
        +\varepsilon^{\alpha} C\norm[\infty]{ T}^2\norm[2]{v_F}^2 .
    \end{gathered}
\end{equation}

\paragraph{Estimate of $I_2$.} Similarly to $I_1$, applying H\"older's inequality and Poincar\'e's inequality, one has that
\begin{equation}
    \label{est:ff-003}
    \begin{aligned}
        I_2 &\leq \varepsilon^2|(\partial_t w_T,w_F)| \leq \varepsilon^2\norm[H^{-1}]{\partial_t w_T}\norm[2]{\grad{w_F}}
        \overset{\eqref{est:bf-101-2}}{\lesssim} \varepsilon^2\norm[H^{-1}]{\partial_t T}\norm[2]{\grad{w_F}} \\
        &\overset{\eqref{est:tmp-103}}{\lesssim} \varepsilon^2\norm[2]{Q}\norm[2]{\grad{w_F}}+\varepsilon^2\norm[2]{\grad{T}}\norm[2]{\grad{w_F}}
         +\varepsilon^2 \norm[\infty]{T}\norm[2]{\grad{T}}\norm[2]{\grad{w_F}}\\
         &\qquad\qquad + \varepsilon^2 \norm[\infty]{T}\norm[2]{v_F,w_F} \norm[2]{\grad{w_F}}.
        % +\varepsilon^2 \norm[6]{T}\underbrace{\norm[3]{v_F, w_F}}_{\mathclap{\lesssim \norm[2]{v_F}^{1/2}\norm[2]{\nabla_3 v_F}^{1/2} + \norm[2]{w_F}^{1/2}\norm[2]{\nabla_3 w_F}^{1/2}}}\norm[2]{\grad_3{w_F}}
    \end{aligned}
\end{equation}
Therefore, applying Young's inequality in \eqref{est:ff-003} implies that, for some constant $ C \in (0,\infty) $,
\begin{equation}
    \label{est:ff-004}
    \begin{gathered}
        I_2 \leq \frac{1}{100} \varepsilon^{-\alpha} \norm[2]{\varepsilon \grad{w_F}}^2 + \varepsilon^{\alpha+2}C\norm[2]{Q}^2+ \varepsilon^{\alpha+2}C\norm[2]{\grad{T}}^2+\varepsilon^{\alpha+2}C\norm[\infty]{T}^2\norm[2]{\grad{T}}^2\\ 
        + \varepsilon^{\alpha+2} \norm[\infty]{T}^2 \norm[2]{v_F, w_F}^2.
    \end{gathered}
\end{equation}

\paragraph{Estimate of $I_3$.} Notice that $ v = v_T + v_F $. We write $I_3$ as
\begin{equation}
    \label{est:ff-005}
    I_3 = (v\cdot\grad_h{v_T},v_F) = (v_T\cdot\grad_h{v_T},v_F)+(v_F\cdot\grad_h{v_T},v_F) =: I_{3,1}+I_{3,2}.
\end{equation}
Applying H\"older's inequality, the Gagliardo-Nirenberg inequality, and Poincar\'e's inequality, one has that
\begin{equation}
    \label{est:ff-006}
    % \begin{gathered}
    I_{3,1} \lesssim \norm[6]{v_T}\norm[2]{\grad{v_T}}\norm[3]{v_F} \lesssim \norm[2]{\grad{v_T}}^2\norm[2]{\grad{v_F}}
    \overset{\eqref{est:bf-101}}{\lesssim} \norm[2]{T}\norm[2]{\grad{T}}\norm[2]{\grad{v_F}}.
    % \end{gathered}
\end{equation}
Similarly, we have
\begin{equation}
    \label{est:ff-007}
    % \begin{gathered}
    I_{3,2} \lesssim\norm[4]{v_F}^2\norm[2]{\nabla v_T} \overset{\eqref{est:bf-101}}{\lesssim} \norm[2]{T}\norm[2]{v_F}^{1/2}\norm[2]{\grad{v_F}}^{3/2}.
    % \end{gathered}
\end{equation}
Therefore, applying Young's inequality in \eqref{est:ff-006} and \eqref{est:ff-007} implies that, for some constant $C\in (0,\infty)$,
\begin{equation}
    \label{est:ff-008}
    I_3 \leq \frac{1}{100} \varepsilon^{-\alpha}\norm[2]{\grad{v_F}}^2 + \varepsilon^\alpha C \norm[2]{T}^2\norm[2]{\grad{T}}^2 + \varepsilon^{3\alpha} C\norm[2]{T}^2\norm[2]{\grad{T}}^2 \norm[2]{v_F}^2,
\end{equation}
where we have also applied Poincar\'e's inequality.

% We have, applying Proposition \ref{bf-est},
% \begin{align*}
%     I_{3} & \leq |(v_T\cdot\grad{v_T},v_F)|+|(v_F\cdot\grad{v_T},v_F)|                                                                                                                                                      \\
%           & \leq \norm[6]{v_T}\norm[2]{\grad_3{v_T}}\norm[3]{v_F}+\norm[4]{v_F}^2\norm[2]{\grad_3{v_T}}                                                                                                                     \\
%           & \lesssim \norm[H^1_x]{v_T}^2\norm[2]{v_F}^\frac{1}{2}\norm[2]{\grad_3{v_F}}^\frac{1}{2}+\norm[2]{T}^\frac{1}{2}\norm[2]{\grad_3{T}}^\frac{1}{2}\norm[2]{v_F}^\frac{1}{2}\norm[2]{\grad_3{v_F}}^\frac{3}{2}      \\
%           & \lesssim \norm[2]{T}\norm[2]{\grad_3{T}}\norm[2]{\grad_3{v_F}}+\norm[2]{T}^\frac{1}{2}\norm[2]{\grad_3{T}}^\frac{1}{2}\norm[2]{v_F}^\frac{1}{2}\norm[2]{\grad_3{v_F}}^\frac{3}{2}                               \\
%           & \leq \varepsilon^\alpha C\norm[2]{T}^2\norm[2]{\grad_3{T}}^2+\varepsilon^{3\alpha}C\norm[2]{T}^2\norm[2]{\grad_3{T}}^2\norm[2]{v_F}^2+\varepsilon^{-\alpha}\frac{\nu_\mathrm{min}}{10}\norm[2]{\grad_3{v_F}}^2.
% \end{align*}

\paragraph{Estimate of $I_4$.} Similarly to \eqref{wt-bound-vt-0}, thanks to $ \eqref{sys:BE}_3 $, one has that
\begin{equation}
    \label{est:ff-009}
    w = - \int_0^z \grad_h{} \cdot v(\zeta) \,d\zeta.
\end{equation}
Therefore, one can write, after applying H\"older's inequality and Minkowski's inequality,
\begin{equation}
    \label{est:ff-010}
    \begin{gathered}
        I_4 = (\int_0^z \grad_h{}\cdot  v(\zeta) \,d\zeta\, \partial_z v_T, v_F) \leq  \int_0^1\int_{\mathbb{T}^2}\biggl(\int_0^1 |\grad_h{}\cdot v(\zeta)|\,d\zeta\biggr)|\partial_z v_T|\,|v_F|\,dxdy\,dz\\
        \leq \int_0^1\left(\int_0^1 \norm[L^2(\mathbb{T}^2)]{\grad{v}(\zeta)}\,d\zeta\right)\norm[L^4(\mathbb{T}^2)]{\partial_z v_T(z)}\norm[L^4(\mathbb{T}^2)]{v_F(z)}\,dz.
    \end{gathered}
\end{equation}
Now, applying the Gagliardo-Nirenberg inequality in $ \mathbb T^2 $ implies that, for each $ z \in \mathbb T $,
\begin{equation}
    \label{est:ff-011}
    \begin{gathered}
        \norm[L^4(\mathbb{T}^2)]{\partial_z v_T(z)}\lesssim \norm[L^2(\mathbb{T}^2)]{\partial_z v_T(z)}^{1/2} \norm[L^2(\mathbb{T}^2)]{\grad_h{\partial_z v_T(z)}}^{1/2}, \\
        \norm[L^4(\mathbb{T}^2)]{v_F(z)} \lesssim \norm[L^2(\mathbb{T}^2)]{v_F(z)}^{1/2} \norm[L^2(\mathbb{T}^2)]{\grad_h{v_F(z)}}^{1/2}.
    \end{gathered}
\end{equation}
Thus, substituting \eqref{est:ff-011} into \eqref{est:ff-010} yields, after applying H\"older's inequality,
\begin{equation}
    \label{est:ff-012}
    \begin{aligned}
        I_4 & \lesssim \norm[2]{\grad{v}} \int_0^1  \norm[L^2(\mathbb{T}^2)]{\partial_z v_T(z)}^{1/2} \norm[L^2(\mathbb{T}^2)]{\grad_h{\partial_z v_T(z)}}^{1/2} \norm[L^2(\mathbb{T}^2)]{v_F(z)}^{1/2} \norm[L^2(\mathbb{T}^2)]{\grad_h{v_F(z)}}^{1/2}\,dz                 \\
            & \lesssim \norm[2]{\grad{v_T}}\norm[2]{\partial_z v_T}^\frac{1}{2}\norm[2]{\grad{\partial_z v_T}}^\frac{1}{2}\norm[2]{v_F}^\frac{1}{2}\norm[2]{\grad{v_F}}^\frac{1}{2}                       \\
            & \qquad + \norm[2]{\grad{v_F}}\norm[2]{\partial_z v_T}^\frac{1}{2}\norm[2]{\grad{\partial_z v_T}}^\frac{1}{2}\norm[2]{v_F}^\frac{1}{2}\norm[2]{\grad{v_F}}^\frac{1}{2} =: I_{4,1} + I_{4,2},
    \end{aligned}
\end{equation}
where we have used the fact that $ v = v_T + v_F $.
To estimate $ I_{4,1} $ and $ I_{4,2} $, applying \eqref{est:bf-101} and \eqref{est:bf-102} leads to
\begin{equation}
    \label{est:ff-013}
    \begin{aligned}
        I_{4,1} & \lesssim \norm[2]{T}^{2} \norm[2]{v_F}^{1/2} \norm[2]{\grad{v_F}}^{1/2}, \\
        I_{4,2} & \lesssim \norm[2]{T}\norm[2]{v_F}^{1/2}\norm[2]{\grad{v_F}}^{3/2}.
    \end{aligned}
\end{equation}
Therefore, substituting \eqref{est:ff-013} into \eqref{est:ff-012} and applying Young's inequality implies that, for some constant $ C \in(0,\infty) $,
\begin{equation}
    \label{est:ff-014}
    % \begin{aligned}
    I_4 \leq \frac{1}{100} \varepsilon^{-\alpha} \norm[2]{\grad{v_F}}^2 + \varepsilon^\alpha C \norm[2]{T}^2 \norm[2]{\grad{T}}^2 + \varepsilon^{3\alpha}  \norm[2]{T}^2 \norm[2]{\grad{T}}^2 \norm[2]{v_F}^2,
    % \end{aligned}
\end{equation}
where we have also applied Poincar\'e's inequality.

% \begin{equation}
%     \begin{gathered}
%         \lesssim (\norm[2]{\grad_3{v_T}}+\norm[2]{\grad_3{v_F}})\norm[2]{T}\norm[2]{v_F}^\frac{1}{2}\norm[2]{\grad_3{v_F}}^\frac{1}{2}                                                                                                                                \\
%         \lesssim \norm[2]{T}\norm[2]{\grad_3{T}}\norm[2]{\grad_3{v_F}}+\norm[2]{T}^\frac{1}{2}\norm[2]{\grad_3{T}}^\frac{1}{2}\norm[2]{v_F}^\frac{1}{2}\norm[2]{\grad_3{v_F}}^\frac{3}{2}                                                                             \\
%         \leq \varepsilon^\alpha C\norm[2]{T}^2\norm[2]{\grad_3{T}}^2+\varepsilon^{3\alpha}C\norm[2]{T}^2\norm[2]{\grad_3{T}}^2\norm[2]{v_F}^2+\varepsilon^{-\alpha}\frac{\nu_\mathrm{min}}{10}\norm[2]{\grad_3{v_F}}^2.
%     \end{gathered}
% \end{equation}

\paragraph{Estimate of $I_5$.} We write $ I_5 $ as
\begin{equation}
    \label{est:ff-015}
    I_5 = - \varepsilon^2(v_T \cdot \grad_h{w_T}, w_F) - \varepsilon^2 (v_F\cdot \grad_h{w_T}, w_F) =: I_{5,1} + I_{5,2}.
\end{equation}
Using \eqref{wt-bound-vt-0}, one can write, after applying H\"older's inequality and Minkowski's inequality,
\begin{equation}
    \label{est:ff-016}
    \begin{gathered}
        I_{5,1} = \varepsilon^2 (v_T\cdot \grad_h{}
        (\int_0^z \grad_h{}\cdot v_{T}(\zeta) \,d\zeta),w_F) \lesssim \varepsilon^2 \int_0^1 \int_{\mathbb{T}^2}|v_T(z)|\left(\int_0^1 |\grad_h{}\grad_h{}\cdot v_T(\zeta)|\,d\zeta\right)|w_F(z)|\,dxdy\,dz \\
        \lesssim \varepsilon^2 \norm[2]{\nabla^2 v_T} \int_0^1 \norm[L^4(\mathbb T^2)]{v_T(z)} \norm[L^4(\mathbb T^2)]{w_F(z)} \,dz,
    \end{gathered}
\end{equation}
Analogously to \eqref{est:ff-011} and \eqref{est:ff-012}, applying the two dimensional Gagliardo-Nirenberg inequality and H\"older's inequality to the right of \eqref{est:ff-016} yields
\begin{equation}
    \label{est:ff-017}
    \begin{aligned}
        I_{5,1} & \lesssim \varepsilon^2 \norm[2]{\nabla^2 v_T} \int_0^1 \norm[L^2(\mathbb T^2)]{v_T(z)}^{1/2} \norm[L^2(\mathbb T^2)]{\grad_h{v_T(z)}}^{1/2} \norm[L^2(\mathbb T^2)]{w_F(z)}^{1/2} \norm[L^2(\mathbb T^2)]{\grad_h{w_F(z)}}^{1/2} \,dz \\
                & \lesssim \varepsilon^2 \norm[2]{\nabla^2 v_T}  \norm[2]{v_T}^{1/2} \norm[2]{\nabla v_T}^{1/2} \norm[2]{w_F}^{1/2} \norm[2]{\nabla w_F}^{1/2}                                                                                            \\
                & \overset{\eqref{est:bf-101}}{\lesssim} \varepsilon^2 \norm[2]{T}\norm[2]{\nabla T}\norm[2]{w_F}^{1/2} \norm[2]{\nabla w_F}^{1/2}.
    \end{aligned}
\end{equation}
Similarly, one has that
\begin{equation}
    \label{est:ff-018}
    \begin{aligned}
        I_{5,2} \lesssim \varepsilon^2 \norm[2]{\nabla T} \norm[2]{v_F}^{1/2} \norm[2]{\nabla v_F}^{1/2} \norm[2]{w_F}^{1/2} \norm[2]{\nabla w_F}^{1/2}.
    \end{aligned}
\end{equation}
Therefore, substituting \eqref{est:ff-017} and \eqref{est:ff-018} into \eqref{est:ff-015} and applying Poincar\'e's inequality and Young's inequality implies that, for some constant $ C \in (0,\infty) $,
\begin{equation}
    \label{est:ff-019}
    % \begin{aligned}
    I_5 \leq \frac{1}{100} \varepsilon^{-\alpha} \norm[2]{\grad{v_F},\varepsilon \grad{w_F}}^2  + \varepsilon^{\alpha+2} C \norm[2]{T}^2 \norm[2]{\grad{T}}^2 + \varepsilon^{\alpha+2} C \norm[2]{\grad{T}}^2 \norm[2]{v_F,\varepsilon w_F}^2.
    % \end{aligned}
\end{equation}

\paragraph{Estimate of $I_6$.} We write $ I_6 $ as
\begin{equation}
    \label{est:ff-020}
    I_6 = - \varepsilon^2 (w_T \partial_z w_T, w_F) - \varepsilon^2 (w_F \partial_z w_T, w_F) =: I_{6,1} + I_{6,2}.
\end{equation}
Using \eqref{wt-bound-vt-0}, one can write, after applying H\"older's inequality, Minkowski's inequality, and the two dimensional Gagliardo-Nirenberg inequality,
\begin{equation}
    \label{est:ff-021}
    \begin{gathered}
        I_{6,1} = - \varepsilon^2 (\int_0^z \grad_h{} \cdot v_T(\zeta)\,d\zeta\, \grad_h{} \cdot v_T, w_F) \lesssim \varepsilon^2 \int_0^1 \int_{\mathbb{T}^2}\left(\int_0^1 |\grad_h{}\cdot v_T(\zeta)|\,d\zeta\right)|\grad_h{}\cdot v_T(z)|\,|w_F(z)|\,dxdy\,dz \\
        \lesssim \varepsilon^2 \norm[2]{\nabla v_T} \int_0^1 \norm[L^4(\mathbb T^2)]{\grad_h{v_T(z)}} \norm[L^4(\mathbb T^4)]{w_F(z)} \,dz \\
        \lesssim \varepsilon^2 \norm[2]{\nabla v_T} \int_0^1 \norm[L^2(\mathbb T^2)]{\grad_h{v_T(z)}}^{1/2} \norm[L^2(\mathbb T^2)]{\grad_h^2{v_T(z)}}^{1/2} \norm[L^2(\mathbb T^2)]{w_F(z)}^{1/2} \norm[L^2(\mathbb T^2)]{\grad_h{w_F(z)}}^{1/2}\,dz \\
        \lesssim \varepsilon^2 \norm[2]{\nabla v_T} \norm[2]{\nabla v_T}^{1/2} \norm[2]{\nabla^2 v_T}^{1/2} \norm[2]{w_F}^{1/2} \norm[2]{\nabla w_F}^{1/2} \\
        \overset{\eqref{est:bf-101}}{\lesssim} \varepsilon^2 \norm[2]{T}^{3/2} \norm[2]{\grad{T}}^{1/2} \norm[2]{w_F}^{1/2} \norm[2]{\nabla w_F}^{1/2}.
    \end{gathered}
\end{equation}
Similar arguments yield that
\begin{equation}
    \label{est:ff-022}
    \begin{gathered}
        I_{6,2} \lesssim \varepsilon^2 \norm[2]{T}^{1/2} \norm[2]{\nabla T}^{1/2} \norm[2]{\nabla v_F} \norm[2]{w_F}^{1/2} \norm[2]{\nabla w_F}^{1/2}.
    \end{gathered}
\end{equation}
Therefore, substituting \eqref{est:ff-021} and \eqref{est:ff-022} into \eqref{est:ff-020} and applying Poincar\'e's inequality and Young's inequality implies that, for some constant $C\in (0,\infty)$,
\begin{equation}
    \label{est:ff-023}
    I_6 \leq \frac{1}{100} \varepsilon^{-\alpha} \norm[2]{\grad{v_F},\varepsilon\grad{w_F}}^2  + \varepsilon^{\alpha+2} C \norm[2]{T}^2 \norm[2]{\grad{T}}^2 + \varepsilon^{3 \alpha+4} C \norm[2]{T}^2 \norm[2]{\grad{T}}^2 \norm[2]{\varepsilon w_F}^2.
\end{equation}

\paragraph{Estimate of $I_7$.} Applying H\"older's inequality, one has that
\begin{equation}
    \label{est:ff-024}
    \begin{gathered}
        I_7\leq \varepsilon^{2-\alpha}|(\grad{w_T},\grad{w_F})| \leq \varepsilon^{2-\alpha} \norm[2]{\grad{w_T}} \norm[2]{\grad{w_F}} \overset{\eqref{est:bf-101}}{\lesssim} \varepsilon^{2-\alpha} \norm[2]{\grad{T}} \norm[2]{\grad{w_F}}.
    \end{gathered}
\end{equation}
Therefore, applying Young's inequality in \eqref{est:ff-024} implies that, for some constant $ C \in (0,\infty) $,
\begin{equation}
    \label{est:ff-025}
    I_7 \leq \frac{1}{100} \varepsilon^{-\alpha} \norm[2]{\varepsilon\grad{w_F}}^2 + \varepsilon^{2-\alpha} C \norm[2]{\grad{T}}^2.
\end{equation}

\paragraph{Summary of estimates.}

Collecting \eqref{est:ff-000}--\eqref{est:ff-025}, we can conclude that, for some constant $ C \in (0,\infty) $, 
\begin{equation}
    \label{est:ff-101}
    \begin{gathered}
        \dfrac{d}{dt} \norm[2]{v_F,\varepsilon w_F}^2 + \varepsilon^{-\alpha} \norm[2]{\grad{v_F}, \varepsilon \grad{w_F}}^2 \leq \varepsilon^{\alpha} C \norm[\infty]{T}^2 \underbrace{\norm[2]{v_F,\varepsilon w_F}^2}_{\mathclap{\lesssim \norm[2]{\grad{v_F},\varepsilon\grad{w_F}}^2}}\\
        + \varepsilon^\alpha C\norm[2]{Q}^2+\varepsilon^{\min\lbrace \alpha,2-\alpha\rbrace} C (1+\norm[2]{T}^2 + \norm[\infty]{T}^2) \norm[2]{\grad{T}}^2 \\
        + \varepsilon^{\min\lbrace 3\alpha,\alpha+2\rbrace} C ( 1+ \norm[2]{T}^2 + \norm[\infty]{T}^2)\norm[2]{\grad{T}}^2 \norm[2]{v_F,\varepsilon w_F}^2,
    \end{gathered}
\end{equation}
where we have applied Poincar\'e's inequality. 

Now we consider $ \varepsilon_0 $ such that, thanks to \eqref{est:tmp-101},
\begin{equation}
    \label{est:ff-102}
    \varepsilon_0^{2\alpha} C \norm[\infty]{T}^2 \leq \varepsilon_0^{2\alpha} C (\norm[\infty]{T_0}+\int_0^\infty \norm[\infty]{Q(t)}\,dt)^2  \leq 1/2.
\end{equation}
Then for all $ \varepsilon \in (0,\varepsilon_0) $, 
applying Gr\"onwall's inequality in \eqref{est:ff-101} implies that, for some $ \mathfrak c_1 \in (0,\infty) $, independent of $ \varepsilon $ and the solution, 
\begin{equation}
    \label{est:ff-103}
    \begin{gathered}
    \sup_{0\leq t < \infty} \norm[2]{v_F(t),\varepsilon w_F(t)}^2 + \varepsilon^{-\alpha} \int_0^\infty \norm[2]{\grad{v_F(t)}, \varepsilon\grad{w_F(t)}}^2 \,dt \leq e^{\varepsilon^{\min\lbrace 3\alpha, \alpha+2 \rbrace}\mathfrak c_1 A} \norm[2]{v_{F,0},\varepsilon w_{F,0}}^2 \\
    + \varepsilon^\alpha \mathfrak{c}_1e^{\varepsilon^{\min\{3\alpha,\alpha+2\}}\mathfrak{c}_1 A}\int_0^\infty \norm[2]{Q(t)}^2\,dt+\varepsilon^{\min\lbrace \alpha, 2 -\alpha \rbrace} \mathfrak c_1 e^{\varepsilon^{\min\lbrace 3 \alpha, \alpha+2\rbrace}\mathfrak c_1 A} A,
\end{gathered}
\end{equation}
where, thanks to \eqref{est:tmp-101} and \eqref{est:tmp-102},
\begin{equation}
    \label{est:ff-104}
    \begin{aligned}
        A &:= \int_0^\infty ( 1 + \norm[2]{T(t)}^2 + \norm[\infty]{T(t)}^2)\norm[2]{\grad{T(t)}}^2  \,dt\\
        &\leq \lbrack 1+\norm[2]{T_0}^2
        + \int_0^\infty\norm[2]{Q(t)}^2\,dt+2\norm[\infty]{T_0}^2 + 2(\int_0^\infty \norm[\infty]{Q(t)}\,dt)^2 \rbrack \\
        &\qquad\qquad\times (\norm[2]{T_0}^2 + \int_0^\infty \norm[2]{Q(t)}^2 \,dt) < \infty. 
    \end{aligned}
\end{equation}

\begin{comment}
    \begin{equation}
        \label{est:ff-104}
        \begin{gathered}
            A := \int_0^\infty ( 1 + \norm[2]{T(t)}^2 + \norm[\infty]{T(t)}^2)\norm[2]{\grad{T(t)}}^2  \,dt \leq \lbrack 1+\norm[2]{T_0}^2
            + \int_0^\infty\norm[2]{Q(t)}^2\,dt+2\norm[\infty]{T_0}^2 + 2(\int_0^\infty \norm[\infty]{Q(t)}\,dt)^2 \rbrack \\
            \times (\norm[2]{T_0}^2 + \int_0^\infty \norm[2]{Q(t)}^2 \,dt) < \infty. 
        \end{gathered}
    \end{equation}
\end{comment}

\section{Proof of theorems \ref{thm:uniform-weak-sol} and \ref{thm:pge-limit}}
\label{sec:proof_of_thm}
We have obtained formally the energy inequality \eqref{est:ff-101}. In this section, we
outline the construction of solutions and prove theorems \ref{thm:uniform-weak-sol} and \ref{thm:pge-limit}.

\subsection{The approximating system and existence of global weak solution with uniform-in-$\varepsilon$ estimate}
\label{subsec:approximation}

Let $ N $ be a positive integer. Define the scalar-valued finite dimensional space
\begin{equation}
    \label{def:finite-d-functional-space}
    V_N := \left\{ f \in L^2(\mathbb T^3)\;:\;\int_{\mathbb{T}^3}f\,d\mathbf{x}=0, \ \text{and} \ \bigl(f,e^{i(k_1 x + k_2 y + k_3 z)}\bigr) = 0 \ \text{for all} \ \vert k_1 \vert + \vert k_2 \vert + \vert k_3 \vert > N \right\}.
\end{equation}
Let \begin{equation}
    \label{projection}
    \mathrm{Proj}_N : L^2(\mathbb T^3) \to V_N
\end{equation}
be the $ L^2 $-projection of $ L^2(\mathbb{T}^3)$ onto $ V_N $, with the convention that $\text{Proj}_N$ acts on vector-valued functions component-wise. For arbitrary $ N \in \mathbb Z^+ $, consider the following approximation system:
\begin{equation}
    \label{sys:N-approximation}
    \begin{cases}
        \grad_h{p_{T,N}} + v_{T,N}^\perp - \Delta v_{T,N} = 0, \\
        \partial_z p_{T,N} + T_{N} = 0,\\
        \grad_h{}\cdot v_{T,N} + \partial_z w_{T,N}=0, \\
        % \\
        \partial_t v_{F,N} + \mathrm{Proj}_N (v_{F,N} \cdot \grad_h{v_{F,N}}) + \mathrm{Proj}_N (w_{F,N} \partial_z v_{F,N}) + \frac{1}{\varepsilon^{\alpha}}(\grad_h{p_{F,N}} + v_{F,N}^\perp - \Delta v_{F,N}) \\
        \qquad   = - \partial_t v_{T,N} - \mathrm{Proj}_N (v_N \cdot \grad_h{v_{T,N}}) - \mathrm{Proj}_N (v_{T,N} \cdot \grad_h{v_{F,N}}) - \mathrm{Proj}_N (w_{N} \partial_z v_{T,N}) - \mathrm{Proj}_N (w_{T,N} \partial_z v_{F,N}),\\
        \partial_t w_{F, N} + \mathrm{Proj}_N (v_{F,N} \cdot \grad_h{w_{F,N}}) + \mathrm{Proj}_N (w_{F,N} \partial_z w_{F,N}) + \frac{1}{\varepsilon^{\alpha+2}} \partial_z p_{F,N} - \frac{1}{\varepsilon^{\alpha}} \Delta w_N \\
        \qquad = - \partial_t w_{T,N} - \mathrm{Proj}_N (v_N \cdot \grad_h{w_{T,N}}) - \mathrm{Proj}_N (v_{T,N} \cdot \grad_h{w_{F,N}}) - \mathrm{Proj}_N (w_N \partial_z w_{T,N}) - \mathrm{Proj}_N (w_{T,N} \partial_z w_{F,N}),\\
        \grad_h{}\cdot v_{F,N} + \partial_z w_{F,N} = 0, \\
        \partial_t T_N  + \mathrm{Proj}_N (v_{F,N} \cdot \grad_h{T_N}) + \mathrm{Proj}_N (w_{F,N} \partial_z T_N) - \Delta T_N \\
        \qquad = Q_N - \mathrm{Proj}_N (v_{T,N} \cdot \grad_h{T_N}) - \mathrm{Proj}_N (w_{T,N} \partial_z T_N),
    \end{cases}
\end{equation}
where $ v_N = v_{T,N} + v_{F,N}, \ w_N = w_{T,N} + w_{F,N} $, $ p_{T,N} $, $ p_{F,N} $, and $ T_N $ are in $ V_N $, and $ Q_N := \mathrm{Proj}_N Q $. Then for any $ N $, one can construct a global weak solution to system \eqref{sys:N-approximation} using the Galerkin scheme, and all estimates in sections \ref{sec:well-posedness} and \ref{sec:est-freeflow} hold for $ (v_{F,N}, w_{F,N}, T_N) $, uniformly-in-$ N, \varepsilon $. Sending $ N \rightarrow \infty $ recovers the {\it a priori} estimates in the aforementioned sections.

\smallskip In particular, for $0<\alpha<2$, there exists $\varepsilon_0\in (0,1)$ and some constant $C\in (0,\infty)$, both independent of $\varepsilon$ and depending only on the initial data and $Q$, such that for all $ \varepsilon \in (0,\varepsilon_0) $,
\begin{equation}
\label{est:unifrom-in-varepsilon-spatial}
\begin{gathered}
\sup_{0\leq t < \infty} \bigl\lbrace \norm[2]{v(t),\varepsilon w(t), v_F(t), \varepsilon w_F(t), T(t), w_T(t)}^2 + \norm[H^1]{v_T(t), \partial_z v_T(t)}^2 + \norm[\infty]{T(t)}^2 \bigr\rbrace \\
 + \int_0^\infty \bigl\lbrace \norm[H^1]{v(t),\varepsilon^{-\alpha/2} v_F(t),\varepsilon^{1-\alpha/2} w_F(t), T(t), w_T(t)}^2 + \norm[H^2]{v_T(t)}^2 + \norm[2]{w(t),\varepsilon^{-\alpha/2}w_F(t)}^2 \bigr\rbrace \,dt < C < \infty,
\end{gathered}
\end{equation}
where we have used \eqref{est:ff-009} to calculate the regularity of $ w, w_F $ from that of $ v, v_F $. 
Thus we have obtained a global weak solution to the intermediate system \eqref{sys:BF}--\eqref{sys:FF}, equivalently, the Boussinesq system \eqref{sys:BE}, satisfying estimate \eqref{est:unifrom-in-varepsilon-spatial}. This finishes the proof of theorem \ref{thm:uniform-weak-sol}.

\subsection{Compactness and the planetary geostrophic approximation}
\label{subsec:compactness}

Thanks to \eqref{est:unifrom-in-varepsilon-spatial}, we have that
\begin{equation}
    \label{est:unifrom-spatial-space}
    \begin{gathered}
        v, v_F, T, w_T \in L^\infty(0,\infty;L^2(\mathbb T^3))\cap L^2(0,\infty;H^1(\mathbb T^3)),\quad v_T \in L^\infty(0,\infty;H^1(\mathbb{T}^3))\cap L^2(0,\infty;H^2(\mathbb T^3)),\\
        w,w_F  \in L^2(0,\infty;L^2(\mathbb T^3))
    \end{gathered}
\end{equation}
uniformly-in-$\varepsilon $. 
In addition, \eqref{est:tmp-103}, \eqref{est:unifrom-in-varepsilon-spatial}, \eqref{est:tmp-101}, and \eqref{est:bf-101-2} imply that %using the equations in systems \eqref{sys:BE} and \eqref{sys:FF}, one has that
\begin{equation}
    \label{est:uniform-temporal-space}
    \partial_t T,\partial_t w_T \in L^2(0,\infty;H^{-1}(\mathbb T^3)), \quad  \partial_t v_T \in L^2(0,\infty; L^2(\mathbb T^3))
\end{equation}
uniformly-in-$\varepsilon$. Therefore, we have the following compactness: for any fixed $ \mathcal T \in (0,\infty) $, upon selecting a subsequence of $ \varepsilon \rightarrow 0 $, 
\begin{align}
\label{limit:weak-1}
    v_T, w_T, T &\rightharpoonup v_p, w_p, T_p & \text{weakly in } & L^2(0,\mathcal T;H^1(\mathbb T^3)),\\
    v_T &\rightharpoonup v_p & \text{weakly in } & L^2(0,\mathcal T; H^2(\mathbb T^3)),\\
    v_T, w_T, T &\overset{*}{\rightharpoonup} v_p, w_p, T_p & \text{weak$*$-ly in } & L^\infty(0,\mathcal T;L^2(\mathbb T^3)),\\
    v_T &\overset{*}{\rightharpoonup} v_p & \text{weak$*$-ly in } & L^\infty(0,\mathcal T;H^1(\mathbb{T}^3)),\\
    v_T, w_T,T & \rightarrow v_p, w_p, T_p & \text{in } & L^2(0,\mathcal T; L^2(\mathbb T^3)), \label{limit:weak-1-5}\\
    v_T & \rightarrow v_p & \text{in } & L^2(0,\mathcal T; H^1(\mathbb T^3)), \label{limit:weak-1-6}
\end{align}
for some $ v_p \in L^\infty(0,\mathcal T;H^1(\mathbb T^3)) \cap L^2(0,\mathcal T;H^2(\mathbb T^3)), \ w_p, T_p \in L^\infty(0,\mathcal T;L^2(\mathbb T^3)) \cap L^2(0,\mathcal T;H^1(\mathbb T^3)) $, where we have applied Aubin-Lions' lemma and the Rellich-Kondrachov theorem to obtain \eqref{limit:weak-1-5} and \eqref{limit:weak-1-6}.
Moreover, from \eqref{est:unifrom-in-varepsilon-spatial}, one has that
\begin{equation}
    \norm[L^2(0,\infty;H^1(\mathbb T^3))]{v_F} + \norm[L^2(0,\infty;L^2(\mathbb T^3))]{w_F} \lesssim \varepsilon^{\alpha/2}.
\end{equation}
Therefore, one has that, upon selecting a subsequence of $ \varepsilon \rightarrow 0 $, 
\begin{align}
    v_F  \rightarrow 0, \quad & v \rightarrow v_p & \text{in } & L^2(0,\mathcal T; H^1(\mathbb T^3)),\\
    \label{limit:strong-4}
    w_F  \rightarrow 0, \quad & w \rightarrow w_p & \text{in } & L^2(0,\mathcal T; L^2(\mathbb T^3)).
\end{align}

One can then verify that, in the sense of distribution, $ (v_p, w_p, T_p) $ is a solution to the planetary geostrophic equations \eqref{sys:PGE}. This finishes the proof of theorem \ref{thm:pge-limit}.

\section{Improved regularity and the global unique strong solution to the Boussinesq system}
\label{sec:H-1-sol}

In theorem \ref{thm:uniform-weak-sol}, we have obtained a global weak solution to system \eqref{sys:BE} for small $ \varepsilon $. However, the regularity is not enough to guarantee the uniqueness for any fixed $ \varepsilon $. The goal of this section is to obtain the unique global strong solution to system \eqref{sys:BE}, equivalently system \eqref{sys:BF}--\eqref{sys:FF}, for small $ \varepsilon $. 

Taking the $ L^2(\mathbb{T}^3)$ inner product of $\eqref{sys:FF}_1 $ with $ - \Delta v_F $ and $ \eqref{sys:FF}_2 $ with $ - \varepsilon^2 \Delta w_F $, respectively, integrating by parts, and adding the resultants together leads to
\begin{equation}
    \label{h-1-est:000}
    \begin{gathered}
        \frac{1}{2}\frac{d}{dt} \norm[2]{\grad{v_F}, \varepsilon \grad{w_F}}^2 + \varepsilon^{-\alpha} \norm[2]{\Delta v_F, \varepsilon \Delta w_F}^2 = (\partial_t v_T, \Delta v_F)  + \varepsilon^2 (\partial_t w_T, \Delta w_F)  \\
         +  (v_F \cdot \grad_h{v_F},  \Delta v_F) +  (w_F\partial_z v_F, \Delta v_F) + \varepsilon^2  (v_F \cdot \grad_h{w_F}, \Delta w_F)  + \varepsilon^2 (w_F \partial_z w_F, \Delta w_F) \\
         + (v \cdot \grad_h{v_T}, \Delta v_F) + (v_T \cdot \grad_h{v_F}, \Delta v_F) + \varepsilon^2  (v \cdot \grad_h{w_T}, \Delta w_F) + \varepsilon^2 (v_T \cdot \grad_h{w_F}, \Delta w_F) \\
         + (w \partial_z v_T, \Delta v_F) +  (w_T \partial_z v_F, \Delta v_F) + \varepsilon^2  (w \partial_z w_T, \Delta w_F) + \varepsilon^2 (w_T \partial_z w_F, \Delta w_F) \\
         - \varepsilon^{2-\alpha} (\Delta w_T , \Delta w_F) =: \sum_{j=1}^{15} J_j.
    \end{gathered}
\end{equation}
Similarly, taking the $ L^2(\mathbb{T}^3)$ inner product of $ \eqref{sys:FF}_4 $ with $ - \Delta T $ and integrating by parts leads to
\begin{equation}
    \label{h-1-est:000-1}
    \begin{gathered}
        \frac{1}{2}\frac{d}{dt}\norm[2]{\grad{T}}^2 + \norm[2]{\Delta T}^2 =  (v_T \cdot \grad_h{T} + w_T \partial_z T,\Delta T)  + (v_F \cdot \grad_h{T} + w_F \partial_z T, \Delta T)  \\ -  (Q, \Delta T) =:\sum_{j=16}^{18} J_j.
    \end{gathered}
\end{equation}

Next, we will estimate the $ J_js $ for $ j = 1,2,\cdots,18 $. 

\paragraph{Estimate of $J_1$.} Applying H\"older's inequality, one has that
\begin{equation}
    \label{h-1-est:001}
    \begin{aligned}
        J_1\leq |(\partial_t v_T,\Delta v_F)| \leq \norm[2]{\partial_t v_T}\norm[2]{\Delta v_F} \overset{\eqref{est:bf-101-2}}{\lesssim} \norm[H^{-1}]{\partial_t T} \norm[2]{\Delta v_F}.
    \end{aligned}
\end{equation}
Therefore, substituting \eqref{est:tmp-103} into \eqref{h-1-est:001} and applying Young's inequality implies that, for some constant $C\in (0,\infty)$,
\begin{equation}
    \label{h-1-est:002}
    \begin{aligned}
        J_1 & \leq \frac{1}{100} \varepsilon^{-\alpha} \norm[2]{\Delta v_F}^2 + \varepsilon^\alpha C \norm[2]{Q}^2+\varepsilon^\alpha C(1+\norm[\infty]{T}^2)\norm[2]{\grad{T}}^2+\varepsilon^{2\alpha}C\norm[\infty]{T}^2(\varepsilon^{-\alpha}\norm[2]{\grad{v_F}}^2),
    \end{aligned}
\end{equation}
where we have used the estimate
\begin{equation}
    \label{h-1-est:003}
    \norm[2]{v_F,w_F} \lesssim \norm[2]{\grad{v_F}}
\end{equation}
thanks to Poincar\'e's inequality and \eqref{est:ff-incomp}.

\paragraph{Estimate of $J_2$.}
Similarly to $J_1$, applying H\"older's inequality and Young's inequality, one has that
\begin{equation}
    \label{h-1-est:004}
    % \begin{aligned}
        J_2  \leq \frac{1}{100}\varepsilon^{-\alpha}\norm[2]{\varepsilon\Delta w_F}^2 + \varepsilon^{\alpha+2} C \norm[2]{\partial_t w_T}^2
         \overset{\eqref{est:bf-101-2}}{\leq} \frac{1}{100}\varepsilon^{-\alpha}\norm[2]{\varepsilon\Delta w_F}^2 + \varepsilon^{\alpha+2} C \norm[2]{\partial_t T}^2
    % \end{aligned}
\end{equation}
for some constant $ C \in (0,\infty) $. Meanwhile, substituting the equation $\eqref{sys:FF}_{4}$ and applying H\"older's inequality, the Gagliardo-Nirenberg inequality, and Poincar\'e's inequality leads to
\begin{equation}
    \label{h-1-est:005}
    \begin{aligned}
         \norm[2]{\partial_t T}^2 & =\norm[2]{Q + \Delta T - v \cdot \grad_h{T} - w\partial_z T}^2 
         \lesssim \norm[2]{Q}^2 + \norm[2]{\Delta T}^2 + \norm[3]{v,w}^2\norm[6]{\grad{T}}^2 \\
        & \lesssim \norm[2]{Q}^2 + \norm[2]{\Delta T}^2 + (\norm[2]{v_F,w_F}\norm[2]{\grad{v_F}, \grad{w_F}}+\norm[2]{v_T,w_T}\norm[2]{\grad{v_T},\grad{w_T}}) \norm[2]{\Delta T}^2 \\
        & \overset{\eqref{est:bf-101}}{\lesssim} \norm[2]{Q}^2 + (1+\norm[2]{\grad{v_F}, \grad{w_F},\grad{T}}^2) \norm[2]{\Delta T}^2.
    \end{aligned}
\end{equation}
Therefore, substituting \eqref{h-1-est:005} into \eqref{h-1-est:004} yields
\begin{equation}
    \label{h-1-est:006}
    \begin{aligned}
        J_2 & \leq \frac{1}{100}\varepsilon^{-\alpha}\norm[2]{\varepsilon\Delta w_F}^2 + \varepsilon^\alpha C \norm[2]{Q}^2 + \varepsilon^\alpha C(1 + \norm[2]{\grad{v_F}, \varepsilon \grad{w_F},\grad{T}}^2) \norm[2]{\Delta T}^2. 
    \end{aligned}
\end{equation}

\paragraph{Estimate of $J_3$.} Applying H\"older's inequality and the Gagliardo-Nirenberg inequality, one has that
\begin{equation}
    \label{h-1-est:007}
    \begin{aligned}
        J_3 &\leq |(v_F\cdot\grad_h{v_F},\Delta v_F)| \leq \norm[6]{v_F}\norm[3]{\nabla v_F} \norm[2]{\Delta v_F}\lesssim \norm[2]{\grad{v_F}}^{3/2}\norm[2]{\Delta v_F}^{3/2}.
    \end{aligned}
\end{equation}
Therefore, applying Young's inequality in \eqref{h-1-est:007} implies that, for some constant $C\in (0,\infty)$,
\begin{equation}
    \label{h-1-est:008}
    J_3 \leq \frac{1}{100}\varepsilon^{-\alpha}\norm[2]{\Delta v_F}^2 + \varepsilon^{4\alpha} C \norm[2]{\grad{v_F}}^4(\varepsilon^{-\alpha}\norm[2]{\grad{v_F}}^2).
\end{equation}

\paragraph{Estimate of $J_4$.} Thanks to the incompressibility $ \eqref{sys:FF}_{3} $, one can write, after applying H\"older's inequality, Minkowski's inequality, and the two dimensional Gagliardo-Nirenberg inequality,
\begin{equation}
    \label{h-1-est:009}
    \begin{aligned}
    J_4 & = - \int_{\mathbb T^2} \int_0^1 \int_0^z \grad_h{}\cdot v_F(\zeta) \,d\zeta\,\partial_z v_F \cdot \Delta v_F \,dz \,dxdy \\
    & \lesssim \bigl(\int_0^1 \norm[L^4(\mathbb T^2)]{\nabla_h v_F(\zeta)}\,d\zeta\bigr)\bigl(\int_0^1 \norm[L^4(\mathbb T^2)]{\partial_z v_F(z)}\norm[L^2(\mathbb T^2)]{\Delta v_F(z)} \,dz\bigr) \\
    & \lesssim \bigl(\int_0^1 \norm[L^2(\mathbb T^2)]{\nabla v_F(\zeta)}^{1/2} \norm[L^2(\mathbb T^2)]{\Delta v_F(\zeta)}^{1/2} \,d\zeta\bigr)\bigl(\int_0^1 \norm[L^2(\mathbb T^2)]{\partial_z v_F(z)}^{1/2}\norm[L^2(\mathbb T^2)]{\Delta v_F(z)}^{3/2}\,dz\bigr).
\end{aligned}
\end{equation}
Therefore, applying H\"older's inequality in \eqref{h-1-est:009} implies that, for some constant $C\in (0,\infty)$,
\begin{equation}
    \label{h-1-est:010}
    \begin{aligned}
        J_4 & \leq C \norm[2]{\nabla v_F}^{1/2} \norm[2]{\Delta v_F}^{1/2}\norm[2]{\partial_z v_F}^{1/2}\norm[2]{\Delta v_F}^{3/2} \leq  \varepsilon^\alpha C\norm[2]{\grad{v_F}} (\varepsilon^{-\alpha}\norm[2]{\Delta v_F}^2).
    \end{aligned}
\end{equation}

\paragraph{Estimate of $J_5$.} Applying H\"older's inequality and the Gagliardo-Nirenberg inequality, one has that
\begin{equation}
    \label{h-1-est:011}
    \begin{aligned}
        J_5 &\leq \varepsilon^2|(v_F\cdot\grad_h{}w_F,\Delta w_F)|\leq \varepsilon^2 \norm[6]{v_F}\norm[3]{\grad_h{w_F}}\norm[2]{\Delta w_F} \\
        &\lesssim \varepsilon^2 \norm[2]{\grad{v_F}} \norm[2]{\grad_h{w_F}}^{1/2} \norm[2]{\Delta w_F}^{3/2}.
    \end{aligned}
\end{equation} 
Using the fact that $ \norm[2]{\grad_h{w_F}} \leq \norm[2]{\nabla^2 v_F} $ in \eqref{h-1-est:011} implies that, for some constant $C\in (0,\infty)$,
\begin{equation}
    \label{h-1-est:012}
    J_5 \leq \varepsilon^{\alpha+1/2}C\norm[2]{\grad{v_F}} (\varepsilon^{-\alpha} \norm[2]{\Delta v_F}^{1/2} \norm[2]{\varepsilon \Delta w_F}^{3/2})\leq  \varepsilon^{\alpha+1/2}C\norm[2]{\grad{v_F}} (\varepsilon^{-\alpha}\norm[2]{\Delta v_F, \varepsilon \Delta w_F}^2).
\end{equation}

\paragraph{Estimate of $J_6$.} Thanks to the incompressibility $ \eqref{sys:FF}_3 $, one can write, after applying H\"older's inequality and the Gagliardo-Nirenberg inequality,
\begin{equation}
    \label{h-1-est:013}
    \begin{aligned}
        J_6 & = - \varepsilon^2 (w_F \grad_h{}\cdot v_F, \Delta w_F) \leq \varepsilon^2 \norm[6]{w_F} \norm[3]{\nabla v_F} \norm[2]{\Delta w_F} \\
        & \lesssim \varepsilon^2 \norm[2]{\grad{w_F}} \norm[2]{\nabla v_F}^{1/2} \norm[2]{\Delta v_F}^{1/2} \norm[2]{\Delta w_F}.
    \end{aligned}
\end{equation}
Therefore, applying Young's inequality in \eqref{h-1-est:013} implies that, for some constant $ C \in (0,\infty) $,
\begin{equation}
    \label{h-1-est:014}
    J_6 \leq \frac{1}{100} \varepsilon^{-\alpha}\norm[2]{\Delta v_F, \varepsilon \Delta w_F}^2 + \varepsilon^{4\alpha} C \norm[2]{\varepsilon\grad{w_F}}^4 (\varepsilon^{-\alpha}\norm[2]{\nabla v_F}^2).
\end{equation}

\paragraph{Estimate of $J_7$.} Similarly as before, after applying H\"older's inequality and the Gagliardo-Nirenberg inequality, one has that
\begin{equation}
    \label{h-1-est:015}
    \begin{aligned}
        J_7 & \leq |(v_T\cdot\grad_h{v_T},\Delta v_F)|+|(v_F\cdot\grad_h{v_T},\Delta v_F)|\\
        &\leq \norm[6]{v_T} \norm[3]{\nabla v_T} \norm[2]{\Delta v_F} +\norm[6]{v_F} \norm[3]{\nabla v_T} \norm[2]{\Delta v_F}\\
        & \lesssim \norm[2]{\grad{v_T}}^{3/2} \norm[2]{\grad^2{v_T}}^{1/2} \norm[2]{\Delta v_F} + \norm[2]{\grad{v_F}} \norm[2]{\grad{v_T}}^{1/2}\norm[2]{\grad^2{v_T}}^{1/2} \norm[2]{\Delta v_F} \\
        & \overset{\eqref{est:bf-101}}{\lesssim} \norm[2]{T}^{3/2}\norm[2]{\grad{T}}^{1/2} \norm[2]{\Delta v_F}+\norm[2]{\grad{v_F}} \norm[2]{T}^{1/2}\norm[2]{\grad{T}}^{1/2} \norm[2]{\Delta v_F}.
    \end{aligned}
\end{equation}
Therefore, applying Poincar\'e's inequality and Young's inequality in \eqref{h-1-est:015} implies that, for some constant $C\in (0,\infty)$,
\begin{equation}
    \label{h-1-est:016}
    \begin{aligned}
        J_7 \leq \frac{1}{100} \varepsilon^{-\alpha}\norm[2]{\Delta v_F}^2 + \varepsilon^\alpha C \norm[2]{\grad{v_F},\grad{T}}^2\norm[2]{\grad{T}}^2.
    \end{aligned}
\end{equation}

\paragraph{Estimate of $J_8$.} Applying H\"older's inequality and the Gagliardo-Nirenberg inequality, one has that 
\begin{equation}
    \label{h-1-est:017}
    \begin{aligned}
        J_8 & \leq |(v_T\cdot\grad_h{v_F},\Delta v_F)|\leq \norm[6]{v_T} \norm[3]{\nabla v_F}\norm[2]{\Delta v_F}\\
        &\lesssim \norm[2]{\grad{v_T}} \norm[2]{\nabla v_F}^{1/2} \norm[2]{\Delta v_F}^{3/2}
        \overset{\eqref{est:bf-101}}{\lesssim} \norm[2]{T} \norm[2]{\nabla v_F}^{1/2} \norm[2]{\Delta v_F}^{3/2}.
    \end{aligned}
\end{equation}
Therefore, applying Poincar\'e's inequality and Young's inequality in \eqref{h-1-est:017} implies that, for some constant $ C \in (0,\infty) $,
\begin{equation}
    \label{h-1-est:018}
    J_8 \leq \frac{1}{100} \varepsilon^{-\alpha }\norm[2]{\Delta v_F}^2 + \varepsilon^{3\alpha} C \norm[2]{\nabla v_F}^2\norm[2]{\grad{T}}^4.
\end{equation}

\paragraph{Estimate of $J_9$.}
Similarly as before, after applying H\"older's inequality and the Gagliardo-Nirenberg inequality, one has that
\begin{equation}
    \label{h-1-est:019}
    \begin{aligned}
    J_9 & \leq \varepsilon^2|(v_T\cdot\grad_h{w_T},\Delta w_F)|+\varepsilon^2|(v_F\cdot\grad_h{w_T},\Delta w_F)|\\
    &\leq \varepsilon^2 \norm[6]{v_T} \norm[3]{\nabla w_T} \norm[2]{\Delta w_F} + \varepsilon^2 \norm[6]{v_F} \norm[3]{\nabla w_T} \norm[2]{\Delta w_F} \\
    & \lesssim \varepsilon^2 \norm[2]{\grad{v_T}} \norm[2]{\nabla w_T}^{1/2} \norm[2]{\Delta w_T}^{1/2} \norm[2]{\Delta w_F} + \varepsilon^2 \norm[2]{\grad{v_F}} \norm[2]{\nabla w_T}^{1/2} \norm[2]{\Delta w_T}^{1/2} \norm[2]{\Delta w_F} \\
    & \overset{\eqref{est:bf-101}}{\lesssim} \varepsilon^2\norm[2]{T} \norm[2]{\grad{T}}^{1/2} \norm[2]{\Delta T}^{1/2} \norm[2]{\Delta w_F} + \varepsilon^2 \norm[2]{\grad{v_F}} \norm[2]{\grad{T}}^{1/2} \norm[2]{\Delta T}^{1/2} \norm[2]{\Delta w_F}.
\end{aligned}
\end{equation}
Therefore, applying Poincar\'e's inequality and Young's inequality in \eqref{h-1-est:019} implies that, for some constant $C\in (0,\infty)$, 
\begin{equation}
    \label{h-1-est:020}
    J_9 \leq \frac{1}{100}(\varepsilon^{-\alpha}\norm[2]{\varepsilon \Delta w_F}^2 + \norm[2]{\Delta T}^2) + \varepsilon^{2\alpha + 4} C \norm[2]{\grad{v_F},\grad{T}}^4 \norm[2]{\grad{T}}^2.
\end{equation}

\paragraph{Estimate of $J_{10}$.} Applying H\"older's inequality and the Gagliardo-Nirenberg inequality, one has that
\begin{equation}
    \label{h-1-est:021}
    \begin{aligned}
        J_{10} & \leq \varepsilon^2|(v_T\cdot\grad_h{w_F},\Delta w_F)|\leq \varepsilon^2 \norm[6]{v_T}\norm[3]{\nabla w_F}\norm[2]{\Delta w_F} \\
        &\lesssim \varepsilon^2 \norm[2]{\grad{v_T}} \norm[2]{\nabla w_F}^{1/2} \norm[2]{\Delta w_F}^{3/2} \overset{\eqref{est:bf-101}}{\lesssim} \varepsilon^2 \norm[2]{T} \norm[2]{\nabla w_F}^{1/2} \norm[2]{\Delta w_F}^{3/2}.
    \end{aligned}
\end{equation}
Therefore, applying Poincar\'e's inequality and Young's inequality in \eqref{h-1-est:021} implies that, for some constant $ C \in (0,\infty) $, 
\begin{equation}
    \label{h-1-est:022}
    J_{10} \leq \frac{1}{100} \varepsilon^{-\alpha} \norm[2]{\varepsilon \Delta w_F}^2 + \varepsilon^{3\alpha} C \norm[2]{\varepsilon \nabla w_F}^2\norm[2]{\grad{T}}^4.
\end{equation}

\paragraph{Estimate of $J_{11}$ and $J_{12}$.} Using \eqref{est:ff-009} and \eqref{wt-bound-vt-0}, one can write, after applying H\"older's inequality, Minkowski's inequality, and the two dimensional Gagliardo-Nirenberg inequality, 
\begin{equation}
    \label{h-1-est:023}
    \begin{aligned}
        J_{11} + J_{12} & = - \int_{\mathbb T^2} \int_0^1 \int_0^z \grad_h{}\cdot v(\zeta)\,d\zeta\,\partial_z v_T \cdot \Delta v_F \,dz \,dxdy \\
        & \qquad  - \int_{\mathbb T^2} \int_0^1 \int_0^z \grad_h{}\cdot v_T(\zeta)\,d\zeta\, \partial_z v_F \cdot \Delta v_F \,dz \,dxdy \\
        & \lesssim \bigl(\int_0^1 \norm[L^4(\mathbb T^2)]{\grad_h{v_T(\zeta)},\grad_h{v_F(\zeta)}}\,d\zeta\bigr)\bigl(\int_0^1 \norm[L^4(\mathbb T^2)]{\partial_z v_T(z)}\norm[L^2(\mathbb T^2)]{\Delta v_F(z)}\,dz\bigr) \\
        & \qquad + \bigl(\int_0^1 \norm[L^4(\mathbb T^2)]{\grad_h{v_T(\zeta)}}\,d\zeta\bigr)\bigl(\int_0^1 \norm[L^4(\mathbb T^2)]{\partial_z v_F(z)}\norm[L^2(\mathbb T^2)]{\Delta v_F(z)} \,dz\bigr) \\
        & \lesssim \int_0^1 \norm[L^2(\mathbb T^2)]{\grad_h{v_T(\zeta)},\grad_h{v_F(\zeta)}}^{1/2}\norm[L^2(\mathbb T^2)]{\grad_h^2{v_T(\zeta)},\grad_h^2{v_F(\zeta)}}^{1/2} \,d\zeta \\
        & \qquad \qquad \times  \int_0^1 \norm[L^2(\mathbb T^2)]{\partial_z v_T(z)}^{1/2}\norm[L^2(\mathbb T^2)]{\grad_h{\partial_z v_T(z)}}^{1/2} \norm[L^2(\mathbb T^2)]{\Delta v_F(z)} \,dz \\
        & \qquad + \int_0^1 \norm[L^2(\mathbb T^2)]{\grad_h{v_T(\zeta)}}^{1/2}\norm[L^2(\mathbb T^2)]{\grad_h^2{v_T(\zeta)}}^{1/2} \,d\zeta \\
        & \qquad \qquad \times  \int_0^1 \norm[L^2(\mathbb T^2)]{\partial_z v_F(z)}^{1/2}\norm[L^2(\mathbb T^2)]{\grad_h{\partial_z v_F(z)}}^{1/2} \norm[L^2(\mathbb T^2)]{\Delta v_F(z)} \,dz \\
        & \lesssim \norm[2]{\grad{v_T}, \grad{v_F}} \norm[2]{\Delta v_T} \norm[2]{\Delta v_F}+\norm[2]{\grad{v_T}, \grad{v_F}}\norm[2]{\Delta v_F}^{1/2} \norm[2]{\Delta v_T}^{1/2} \norm[2]{\Delta v_F}\\
        & \overset{\eqref{est:bf-101}}{\lesssim} \norm[2]{\grad{v_F},T}\norm[2]{\grad{T}}\norm[2]{\Delta v_F}+\norm[2]{\grad{v_F},T}\norm[2]{\grad{T}}^{1/2} \norm[2]{\Delta v_F}^{3/2}.
    \end{aligned}
\end{equation}
Therefore, applying Poincar\'e's inequality and Young's inequality in \eqref{h-1-est:023} implies that, for some constant $C\in (0,\infty)$,
\begin{equation}
    \label{h-1-est:024}
    J_{11} + J_{12} \leq \frac{1}{100} \varepsilon^{-\alpha} \norm[2]{\Delta v_F}^2 + \varepsilon^\alpha C \norm[2]{\grad{v_F},\grad{T}}^2 \norm[2]{\grad{T}}^2+\varepsilon^{3\alpha} C \norm[2]{\grad{v_F},\grad{T}}^4 \norm[2]{\grad{T}}^2.
\end{equation}

\paragraph{Estimate of $J_{13}$ and $J_{14}$.} Using \eqref{wt-bound-vt-0} and thanks to the incompressibility $\eqref{sys:FF}_3$, one can write, after applying H\"older's inequality and the Gagliardo-Nirenberg inequality,
\begin{equation}
    \label{h-1-est:025}
    \begin{aligned}
        J_{13} + J_{14} & = - \varepsilon^2  (w \grad_h{}\cdot v_T, \Delta w_F) - \varepsilon^2  (w_T \grad_h{}\cdot v_F, \Delta w_F) \\
        & \leq \varepsilon^2 \norm[3]{w} \norm[6]{\nabla v_T} \norm[2]{\Delta w_F} + \varepsilon^2 \norm[6]{w_T} \norm[3]{\nabla v_F} \norm[2]{\Delta w_F} \\
        & \lesssim \varepsilon^2 (\norm[2]{w_T, w_F}^{1/2} \norm[2]{\grad{w_T},\grad{w_F}}^{1/2} \norm[2]{\Delta v_T} + \norm[2]{\grad{w_T}} \norm[2]{\nabla v_F}^{1/2} \norm[2]{\Delta v_F}^{1/2}) \norm[2]{\Delta w_F} \\
        & \overset{\eqref{est:bf-101}}{\lesssim} \varepsilon^2 (\norm[2]{w_F,T}^{1/2}\norm[2]{\grad{w_F},\grad{T}}^{1/2}\norm[2]{\grad{T}}+\norm[2]{\grad{T}}\norm[2]{\grad{v_F}}^{1/2}\norm[2]{\Delta v_F}^{1/2}) \norm[2]{\Delta w_F}.
    \end{aligned}
\end{equation}
Therefore, applying Poincar\'e's inequality and Young's inequality in \eqref{h-1-est:025} implies that, for some constant $C\in (0,\infty)$,
\begin{equation}
    \label{h-1-est:026}
    J_{13} + J_{14} \leq \frac{1}{100} \varepsilon^{-\alpha}\norm[2]{\Delta v_F,\varepsilon \Delta w_F}^2 + \varepsilon^{\alpha} C \norm[2]{\varepsilon \grad{w_F}, \grad{T}}^2 \norm[2]{\grad{T}}^2 + \varepsilon^{3\alpha+4} C\norm[2]{\grad{v_F}}^2 \norm[2]{\grad{T}}^4.
\end{equation} 

\paragraph{Estimate of $J_{15}$.} Applying H\"older's inequality, one has that
\begin{equation}
    \label{h-1-est:027}
        J_{15} \leq \varepsilon^{2-\alpha}|(\Delta w_T, \Delta w_F)|\leq \varepsilon^{2-\alpha}\norm[2]{\Delta w_T}\norm[2]{\Delta w_F}\overset{\eqref{est:bf-101}}{\lesssim}\varepsilon^{2-\alpha}\norm[2]{\Delta T}\norm[2]{\Delta w_F}.
\end{equation}
Therefore, applying Young's inequality in \eqref{h-1-est:027} implies that, for some constant $C\in (0,\infty)$,
\begin{equation}
    \label{h-1-est:027-1}
    J_{15}\leq \frac{1}{100} \varepsilon^{-\alpha}\norm[2]{\varepsilon \Delta w_F}^2 + \varepsilon^{2-\alpha}C \norm[2]{\Delta T}^2.
\end{equation}

\paragraph{Estimate of $J_{16}$.} 
Using \eqref{wt-bound-vt-0}, one can write
\begin{equation}
    \label{h-1-est:028}
    J_{16} =  (v_T \cdot \grad_h{T}, \Delta T)  - \int_{\mathbb{T}^2}\int_0^1 \int_0^z \grad_h{} \cdot v_T(\zeta) \,d\zeta\,\partial_z T \Delta T \,dz\,dxdy =: J_{16,1} + J_{16,2}.
\end{equation}
Applying H\"older's inequality and the Gagliardo-Nirenberg inequality, one has that
\begin{equation}
    \label{h-1-est:029}
    \begin{aligned}
        J_{16,1} & \leq \norm[6]{v_T}\norm[3]{\nabla T}  \norm[2]{\Delta T} \lesssim \norm[2]{\grad{v_T}} \norm[2]{\grad{T}}^{1/2} \norm[2]{\Delta T}^{3/2}\overset{\eqref{est:bf-101}}{\lesssim} \norm[2]{T} \norm[2]{\grad{T}}^{1/2} \norm[2]{\Delta T}^{3/2},
    \end{aligned}
\end{equation}
and, additionally applying Minkowski's inequality, 
\begin{equation}
    \label{h-1-est:030}
    \begin{aligned}
        J_{16,2} & \leq \bigl(\int_0^1 \norm[L^4(\mathbb T^2)]{\grad_h{v_T(\zeta)}}\,d\zeta\bigr)\bigl(\int_0^1 \norm[L^4(\mathbb T^2)]{\partial_z T(z)}\norm[L^2(\mathbb T^2)]{\Delta T(z)} \,dz\bigr)  \\
        & \lesssim \int_0^1 \norm[L^2(\mathbb T^2)]{\grad_h{v_T(\zeta)}}^{1/2} \norm[L^2(\mathbb T^2)]{\grad_h^2{v_T(\zeta)}}^{1/2}\,d\zeta\\
        & \qquad \qquad \times\int_0^1 \norm[L^2(\mathbb T^2)]{\partial_z T(z)}^{1/2} \norm[L^2(\mathbb T^2)]{\grad_h{\partial_z T(z)}}^{1/2} \norm[L^2(\mathbb T^2)]{\Delta T(z)} \,dz \\
        & \lesssim \norm[2]{\nabla v_T}^{1/2}\norm[2]{\nabla^2 v_T}^{1/2} \norm[2]{\partial_z T}^{1/2} \norm[2]{\Delta T}^{3/2} \overset{\eqref{est:bf-101}}{\lesssim} \norm[2]{T}^{1/2} \norm[2]{\grad{T}} \norm[2]{\Delta T}^{3/2}.
    \end{aligned}
\end{equation}
Therefore, applying Poincar\'e's inequality in \eqref{h-1-est:029} and Young's inequality in \eqref{h-1-est:029} and \eqref{h-1-est:030} implies that, for some constant $ C \in (0,\infty) $, 
\begin{equation}
    \label{h-1-est:031}
    J_{16} \leq \frac{1}{100}\norm[2]{\Delta T}^2 + C \norm[2]{T}^2 \norm[2]{\grad{T}}^4.
\end{equation}

\paragraph{Estimate of $J_{17}$.} Similarly to $J_{16}$, applying H\"older's inequality, Minkowski's inequality, and the Gagliardo-Nirenberg inequality, one has that
\begin{equation}
    \label{h-1-est:032}
    \begin{aligned}
        J_{17} & \lesssim \norm[2]{\grad{v_F}}\norm[2]{\grad{T}}^{1/2} \norm[2]{\Delta T}^{3/2} + \norm[2]{\nabla v_F}^{1/2} \norm[2]{\Delta v_F}^{1/2} \norm[2]{\partial_z T}^{1/2} \norm[2]{\Delta T}^{3/2}.
    \end{aligned}
\end{equation}
Therefore, applying Young's inequality in \eqref{h-1-est:032} implies that, for some constant $ C \in (0,\infty) $,
\begin{equation}
    \label{h-1-est:033}
    J_{17} \leq \frac{1}{100}\norm[2]{\Delta T}^2 + \varepsilon^\alpha C \norm[2]{\grad{v_F}}^2\norm[2]{\grad{T}}^2(\varepsilon^{-\alpha}\norm[2]{\grad{v_F}}^2) + C \norm[2]{\grad{v_F}}^2 \norm[2]{\grad{T}}^2 \norm[2]{\Delta v_F}^2.
\end{equation}

\paragraph{Estimate of $J_{18}$.} Applying H\"older's inequality and Young's inequality implies that, for some constant $C\in (0,\infty)$, 
\begin{equation}
    \label{h-1-est:034}
    J_{18} \leq |(Q,\Delta T)|\leq \norm[2]{Q}\norm[2]{\Delta T}\leq \frac{1}{100}\norm[2]{\Delta T}^2 + C \norm[2]{Q}^2.
\end{equation}

\paragraph{Summary of $H^1$ estimates.} 

Denote by
\begin{equation}
    \label{def:h-1-energy}
    \mathcal E_1(t):= \norm[2]{\grad{v_F(t)},\varepsilon \grad{w_F(t)},\grad{T}}^2,
\end{equation}
and
\begin{equation}
\label{def:h-1-dissipation}
    \mathcal D_1(t) := \varepsilon^{-\alpha}\norm[2]{\Delta v_F(t), \varepsilon \Delta w_F(t)}^2 + \norm[2]{\Delta T(t)}^2.
\end{equation}
% In addition, 
% \begin{equation}
%     \label{def:total_energy_dissipation}
%     \mathcal E_1 := \mathcal E_{1,m} + \mathcal E_{1,t},\qquad \mathcal D_1 := \mathcal D_{1,m} + \mathcal D_{1,t}.
% \end{equation}
Collecting \eqref{h-1-est:000}--\eqref{h-1-est:034} and applying Young's inequality, we can conclude that, for some constant $C\in (0,\infty)$,
\begin{equation}
    \label{h-1-est:101}
    \begin{gathered} 
        \dfrac{d}{dt}\mathcal E_1(t) + \mathcal D_1(t) \leq 
        \varepsilon^\alpha C (1 + \mathcal E_1^2(t)) \mathcal D_1(t) + \varepsilon^{2-\alpha} C \mathcal D_1(t) 
        %+ \varepsilon^\alpha C (1 + \norm[\infty]{T}^2) \norm[2]{\nabla_3 T}^2 
        \\
        + \varepsilon^{\alpha} C (1+\norm[\infty]{T}^2 + \mathcal E_1^2(t)) ( \varepsilon^{-\alpha} \norm[2]{\grad{v_F},\varepsilon \grad{w_F}}^2 + \norm[2]{\grad{T}}^2) \\
        + C \norm[2]{T}^2 \norm[2]{\grad{T}}^2 \mathcal E_1(t) +  C \norm[2]{Q}^2.
    \end{gathered}
\end{equation}
Finally, applying Gr\"onwall's inequality in \eqref{h-1-est:101} implies that, for some $\mathfrak c_2\in (0,\infty)$, independent of $\varepsilon$ and the solution,
\begin{equation}
    \label{h-1-est:102}
    \begin{gathered}
    \mathcal E_1(t) + \int_0^t \mathcal D_1(s)\,ds \leq B(t)  \mathcal E_1(0) + \varepsilon^{\min\lbrace \alpha,2-\alpha\rbrace} \mathfrak c_2 B(t) \bigl(1 + \sup_{0\leq s \leq t}\mathcal E_1^2(s)\bigr) \int_0^t \mathcal D_1(s)\,ds \\
        + \varepsilon^\alpha \mathfrak c_2 B(t) \sup_{0\leq s\leq t}\bigl(1+\norm[\infty]{T(s)}^2 + \mathcal E_1^2(s)\bigr)\int_0^t (\varepsilon^{-\alpha} \norm[2]{\grad{v_F(s)},\varepsilon \grad{w_F(s)}}^2 + \norm[2]{\grad{T(s)}}^2) \,ds \\
         + \mathfrak c_2 B(t) \int_0^t \norm[2]{Q(s)}^2 \,ds, 
\end{gathered}
\end{equation}
where 
\begin{equation}
    \label{h-1-est:103}
    B(t):= e^{\mathfrak c_2 \int_0^t \Vert{T(s)}\Vert_{2}^2 \Vert{\grad{T(s)}}\Vert_{2}^2 \,ds}.
\end{equation}
Thanks to \eqref{est:unifrom-in-varepsilon-spatial}, one can conclude from \eqref{h-1-est:102}--\eqref{h-1-est:103} that, for $ 0 < \alpha < 2 $, there exists $ \varepsilon_1 \in (0,\varepsilon_0) $ and some constant $ C \in (0,\infty) $, both independent of $\varepsilon$ and depending only the initial data and $Q$, such that for all $ \varepsilon \in (0,\varepsilon_1) $, 
\begin{equation}
    \label{h-1-est:104}
    \sup_{0\leq t < \infty} \norm[H^1]{v_F,\varepsilon w_F, T}^2 + \int_0^\infty \norm[H^2]{\varepsilon^{-\alpha/2}v_F(t),\varepsilon^{1-\alpha/2}w_F(t),T(t)}^2 \,dt < C < \infty. 
\end{equation}
Together with \eqref{est:bf-101}, this finishes the proof of \eqref{thmest:strong-global-sol}. To finish the proof of theorem \ref{thm:global-strong}, one establishes the uniqueness of the strong solution by repeating the arguments in \cite{constantin_navier-stokes_1989}, which is standard and omitted here. 

\section*{Acknowledgments}
The authors would like to thank Professor Edriss Titi for bringing the problem to their attention.

\bibliographystyle{plain} % We choose the "plain" reference style
\bibliography{references.bib}

\end{document}